\documentclass[12pt]{amsart}

\usepackage{amsmath,amssymb}
\usepackage{enumerate}
\usepackage{color}
\usepackage{version}

\newtheorem{Thm}{Theorem}[section]
\newtheorem{Lem}[Thm]{Lemma}

\newtheorem{Cor}[Thm]{Corollary}

\newtheorem{Prop}[Thm]{Proposition}
\newtheorem{Prob}[Thm]{Problem}

\newtheorem{``Conj"}[Thm]{``Conjecture"}

\newtheorem{Ass}[Thm]{Assumption}
\theoremstyle{definition}
\newtheorem{Rem}[Thm]{Remark}
\newtheorem{Ex}[Thm]{Example}
\newtheorem{Def}[Thm]{Definition}

\allowdisplaybreaks[1]

\newcommand{\R}{\ensuremath{\mathbb{R}}}
\newcommand{\C}{\ensuremath{\mathbb{C}}}
\newcommand{\G}{\ensuremath{\mathbb{G}}}
\newcommand{\Z}{\ensuremath{\mathbb{Z}}}
\newcommand{\Q}{\mathbb{Q}}

\newcommand{\D}{\mathcal{D}}

\newcommand{\OO}{\mathcal{O}}
\def\Tr{\mathop{\mathrm{Tr}}\nolimits}

\def\l{\langle}
\def\r{\rangle}

\begin{document}

\title[Fano Shimura varieties]{
Fano Shimura varieties with \\ 
mostly branched 
cusps} 

\author{Yota Maeda, Yuji Odaka}
\address{(Y.M) Sony Group Corporation, 1-7-1 Konan, Minato-ku, Tokyo, 108-0075, Japan/Department of Mathematics, Faculty of Science, Kyoto University, Kyoto 606-8502, Japan}
\email{y.maeda@math.kyoto-u.ac.jp}
\address{(Y.O) Department of Mathematics, Faculty of Science, Kyoto University, Kyoto 606-8502, Japan}
\email{yodaka@math.kyoto-u.ac.jp}

\maketitle

\begin{abstract}
We prove that the 
Satake-Baily-Borel compactification of 
certain Shimura varieties 
are Fano varieties, Calabi-Yau varieties or 
have ample canonical divisors 
with mild singularities. We also prove 
some variants statements, give applications and 
discuss various examples including new ones, for instance, the moduli spaces of unpolarized (log) Enriques surfaces. 
\end{abstract}


\section{Introduction}

We prove that the Satake-Baily-Borel compactification of certain Shimura varieties are 
Fano varieties or with ample canonical divisor by means of special modular forms 
(see Theorem \ref{Fanoness}). 
Their unbranched open subsets are always  quasi-affine, and
in Fano Shimura varieties case, we observe that most of cusps are 
covered by the closure of 
branch divisors. 
In section \ref{Ex.sec}, 
we give various concrete  examples, which 
include the moduli of 
(log) Enriques surfaces, those corresponding to 
$II_{2,26}$, and those  associated to 
various Hermitian lattices which we construct. 

The study of birational types of Shimura varieties 
is a semi-classical topic; Tai \cite{Tai}, Freitag \cite{Freitag} and Mumford \cite{Mumford} (resp. Kond\={o} \cite{KondoI, KondoII},  Gritsenko-Hulek-Sankaran \cite{GHS} and Ma \cite{Ma.gentype}) showed some Siegel (resp. orthogonal) modular varieties are of general type.
Recently, the first author studied a similar problem for unitary modular varieties \cite{Maeda_big}. 

On the other hand, in order to prove that Shimura varieties have negative Kodaira dimension, one of the powerful tools for it is the use of certain reflective modular forms \cite{ 
 Gr, MaYos, GrHu, Gri, Maeda}.
 
 For this recurring theme, our main idea in this paper is to focus on the 
Satake-Baily-Borel compactification, 
study it through modern 
birational geometry adapted 
to singular varieties 
and give 
applications. 
In this paper, we define ``special" reflective modular forms, motivated by the work of Gritsenko-Hulek \cite{GrHu}, and show a criterion for proving 
the Satake-Baily-Borel compactification of 
Shimura varieties are Fano varieties. 

Then, we discuss examples in section \S \ref{Ex.sec}, 
including new ones, 
to which we apply our criterion. 
For instance, it follows that the 
Satake-Baily-Borel compactification of the 
moduli spaces of unpolarized (log) Enriques surfaces are Fano varieties; see Example \ref{Enriques}, \ref{log.Enriques}. 

We also give some applications to 
the understanding of cusps and rationality problems. 
That is, for these 
Fano-like Shimura varieties, 
all but one 
compact cusps are shown to be contained in the 
closure of branch divisors. In the same setup, we also show that 
if 
there are no such compact cusps, 
two general 
points are connected by a 
rational curve  
i.e., rationally 
connected by \cite{Zhang}. See Corollary 
\ref{connected} for 
details. 
The former uses \cite{Amb, Fjn}, 
and in particular it 
logically relies on a 
vanishing theorem proven in 
{\it loc.cit}. We do not 
know of another proof which does not use a  
vanishing theorem 
(Problem \ref{novanishing}). 
See Corollaries 
\ref{connected}, \ref{Qrk1}, 
\ref{-K.big} for the details 
and more assertions proven. For instance, 
the moduli space of (unpolarized) Enriques surface 
is shown to be rationally connected, 
which is a weaker version of a famous result of 
Kond\={o} \cite{Kondo.Enriques}.

\section{Main results and proofs}

In this section, we prove general theorems which are mentioned in the introduction. 
In the later section \S \ref{Ex.sec}, 
we apply them to various concrete examples. First, we introduce some notations. 

\subsection{Convention and Notation}\label{Notation}
Below, we discuss 
the linear equivalence class of a 
Cartier divisor and 
the corresponding holomorphic line bundle 
interchangeably. 
Similarly, 
we do not distinguish the 
$\Q$-linear equivalence class of 
a $\Q$-Cartier divisor 
and the corresponding 
$\Q$-line bundle. 
We use the following notations throughout. 
\begin{itemize}
\item $\mathbb{G}$ is a simple algebraic group over $\Q$, not isogenous to ${\rm SL}(2)$.
\item $G$ is the identity component of $\mathbb{G}(\R)$, which we assume to be a simple Lie group. 
\item $K$ is a maximal compact subgroup of $G$, 
\item The corresponding 
Hermitian symmetric domain is $G/K$. 
\item Take an arithmetic subgroup $\Gamma\subset \mathbb{G}(\Q)$ i.e., 
commensurable to $\G(\Z)$. 
\item $X:=\Gamma\backslash G/K$ and its Satake-Baily-Borel 
compactification $\overline{X}^{\rm SBB}$ (\cite{Sat, BB}). 
\item $\mathbb{H}$ denotes the upper half plane (which is an example of $X$). 
\item 
$\partial 
\overline{X}^{\rm SBB}$ denotes the boundary of 
the Satake-Baily-Borel compactification, i.e., 
$\overline{X}^{\rm SBB}\setminus X$. 

\item Denote a toroidal compactification of 
$X$ in the sense of \cite{AMRT}, 
with an arbitrary fixed cone decompositions, 
simply as $\overline{X}$. (The choice of cone decompositions
 do not affect the following discussions. )
\item Denote the boundary divisor $\overline{X}\setminus X$ as $\Delta$ 
(with coefficients $1$). 
\item Denote the branch divisor of $G/K \to \Gamma\backslash G/K$ 
to be $\cup_{i}B_{i}(\subset X)$ with prime divisors $B_{i}$ and 
branch (or ramification) degree $d_{i}$. We denote the closure of $B_{i}$ in 
$\overline{X}$ (resp., $\overline{X}^{\rm SBB}$) 
as $\overline{B_{i}}$ (resp., $\overline{B_{i}}^{\rm SBB}$). 
\item $X^{o}:=X\setminus \cup_{i}B_{i}$. 
\item $L:=K_{\overline{X}}+\Delta+\sum_{i}\frac{d_{i}-1}{d_{i}}
\overline{B_{i}}\in {\rm Pic}
(\overline{X})
\otimes \mathbb{Q}$ 
and its descended (automorphic)  $\Q$-line bundle on 
$\overline{X}^{\rm SBB}$, i.e.,  $K_{\overline{X}^{\rm SBB}}+ 
\sum_{i}\frac{d_{i}-1}{d_{i}} 
\overline{B_{i}}^{\rm SBB}$. 
\item 
Recall from \cite{BB} and \cite[3.4, 4.2 (also see 1.3)]{Mum77} 
that $L$ is ample (resp., semiample) on 
$\overline{X}^{\rm SBB}$ (resp., $\overline{X}$) and 
a meromorphic section of $L^{\otimes t}$ for $t\in \mathbb{Z}_{>0}$ 
corresponds to meromorphic automorphic form 
of arithmetic weight $ct$ for some $c\in \mathbb{Z}$. 
In this paper, weight always simply refers to the 
arithmetic weight (in the sense of e.g., 
\cite{GHS}) and 
call $c$ the {\it canonical 
weight}, following e.g., \cite{GHS}. 
See also Lemma \ref{can.wt} for the calculation of $c$. 

\end{itemize}

\subsection{Special reflective modular forms}
\label{special.form.section}

Recall that reflective modular form is the concept 
originally formulated in \cite{Gr} for orthogonal case, 
which means that the divisor is defined by reflections. 
In this paper, 
we consider the following stronger properties, 
or proper subclass  
of reflective modular forms. 
The upshot of our general observation 
is that the existence of such special 
reflective modular forms 
give strong implications on the birational properties of  
modular varieties (see Theorem \ref{Fanoness}). 
These modular forms are rare, 
but luckily still various interesting 
examples are known (cf., \cite{Gri}, 
our \S \ref{Ex.sec}). We also construct new  
examples in the section \S \ref{Ex.sec}. 

\begin{Ass}[Special reflective modular forms - General case]
\label{Ass1}
Consider the following subclasses of reflective modular forms.
\begin{enumerate}
\item \label{item1.1} A non-vanishing 
holomorphic section $f$ of 
\[\mathcal{O}_{X}(N(s(X)L-\sum_{i}\frac{d_{i}-1}{d_{i}}B_{i}))
\biggl(:=L^{\otimes aN}\biggl(-\sum_{i}\frac{N(d_i-1)}{d_i}B_i\biggr)\biggr)\]
for some $N\in \Z_{>0}$, 
$s(X)\in \mathbb{Q}_{>0}$ with $s(X)N, \frac{N}{d_i} 
\in 
\mathbb{Z}_{>0}$.
\item \label{item1.2}
A non-vanishing holomorphic section $f$ of 
$\mathcal{O}_{X}(N(s(X)L-\sum_{i}c_{i}B_{i}))$ 
for some $N\in \Z_{>0}$, 
$s(X)\in \mathbb{Q}_{>0}$, and 
$c_{i}\in \mathbb{Q}$ with 
$0\le c_{i}\le \frac{d_{i}-1}{d_{i}}$ for all $i$, 
such that $s(X)N, Nc_{i} \in \mathbb{Z}$. 
\end{enumerate}
We follow 
the same convention below. 
\end{Ass}
For a specific choice of $\G$ and $\Gamma$ that we are about to specify, Assumption \ref{Ass1} (i) specializes to the following simpler condition.

\begin{Ass}[Special reflective modular forms - orthogonal case]
\label{Ass2}
For $n>2$, assume that there is a quadratic lattice $\Lambda$ of signature $(2,n)$ such that $\G=O(\Lambda\otimes \Q)$ with $\Gamma\subset 
O(\Lambda)$. 
In this situation, we consider the 
following subclasses of reflective modular forms.
\begin{enumerate}
\item \label{item2.1}
A non-vanishing 
holomorphic section $f$ of 
$\mathcal{O}_{X}(N(s(X)L-\frac{1}{2}\sum_{i}B_{i}))$ for some $N\in \Z_{>0}$, 
$s(X)\in \mathbb{Q}_{>0}$ with $s(X)N, \frac{N}{2} 
\in 
\mathbb{Z}_{>0}$.. 
\end{enumerate}
\end{Ass}
\noindent
Indeed, for the above  
$\G$ and $\Gamma$, 
Gritsenko-Hulek-Sankaran showed that  every branch divisor arises from 
reflections (of order $2$) 
\cite[2.12, 2.13]{GHS}, i.e., 
the ramification degrees $d_i$ 
are all $2$. 

Note that $N$ is unessential as it gets 
multiplied when replacing $f$ by its power, 
while the quantity $s(X)$ is more essential and 
sometimes called a \textit{slope} in 
the literature. 
When we work on the cases 
$G=O(2,n)$ or $G=U(1,n)$ and regard 
$f$ as a modular form, 
we call its arithmetic weight, 
in the sense of \cite{GHS} for instance, 
simply as a 
weight from now on. 

We also review the following well-known fact 
for the convenience. 
\begin{Lem}[{cf., \cite[Hilfsatz 2.1]{Freitag},  \cite[¥S 6.1]{GHS}}]\label{can.wt}
In the orthogonal case $G=O(2,n)$ 
(resp., in the unitary case $G=U(1,n)$), 
the canonical weight $c$ in the sense of 
\S \ref{Notation} is $n$ (resp., $n+1$). 
\end{Lem}
\begin{proof}
Recall that 
the compact dual $D^c$ 
of $D$ in the orthogonal case 
$G=O(2,n)$ is the $n$-dimensional 
quadratic hypersurface 
(resp., $D^c=\mathbb{P}^{n}$
in the unitary case $G=U(1,n)$), 
its canonical divisor is 
$K_{D^c}=\mathcal{O}_{\mathbb{P}^{n+1}}(-n)
|_{Q^n}$ 
(resp., $K_{D^c}=
\mathcal{O}_{\mathbb{P}^{n}}(-n-1)$) 
so that 
the canonical weight $c$ is $n$ (resp., 
$n+1$). 
\end{proof}
Note that the quantity $s(X)$ 
in Theorem \ref{Fanoness} is the 
(arithmetic) weight of the modular form $s$ 
divided by such canonical weight $c$ and some constant; see Remark \ref{rem:orthog_a} and \ref{unitary_a}. 

Below, we discuss various 
Shimura varieties $X$ which can be roughly 
divided into two types, i.e., 
those with modular forms satisfying 
Assumption \ref{Ass1} (i), and 
those with modular forms satisfying 
Assumption \ref{Ass1} (ii). 

The former is discussed in the 
next subsection \S \ref{Fano.subsec}, with 
examples given in section \S \ref{Ex.sec}, 
and the latter is discussed in 
the subsection \ref{-Kbig.subsec} while some examples are 
given in \cite{GrHu, Maeda}. 

\subsection{Main general results and  proofs}\label{Fano.subsec}

Here is our first general theorem. 

\begin{Thm}[Birational properties]\label{Fanoness}
We follow the notation as above. 
If there is a reflective modular form which satisfies 
Assumption \ref{Ass1} \eqref{item1.1} with some $s(X)\in \mathbb{Q}_{>0}$, then 
the Satake-Baily-Borel compactification $\overline{X}^{\rm SBB}$ of 
$X=\Gamma\backslash D$ only has log canonical singularities and 
$X^o$ is quasi-affine. 
In addition, 
\begin{enumerate}
\item \label{Fano}
if $s(X)>1$, then $\overline{X}^{\rm SBB}$ is a Fano variety i.e., $-K_{\overline{X}^{\rm SBB}}$ is ample 
($\Q$-Cartier),
\item if $s(X)=1$, then $\overline{X}^{\rm SBB}$ is a Calabi-Yau variety i.e., $K_{\overline{X}^{\rm SBB}}\sim_{\Q} 0$,  or 
\item if $s(X)<1$, then \label{lcmodel} $K_{\overline{X}^{\rm SBB}}$ is ample. 
\end{enumerate}
\end{Thm}
\subsubsection*{Terminology}
In this paper, we often 
say a normal variety 
is 
a {\it log canonical model} 
(resp., {\it canonical model}) 
in the sense 
that it only has 
log canonical singularities 
(resp., canonical singularities) 
and the canonical class is ample. 
Hence, in 
the case \eqref{lcmodel} 
above, 
$\overline{X}^{\rm SBB}$
is a log canonical model. 
For the basics of birational geometry, 
we refer to e.g., 
\cite{KM}. 
\begin{proof}
Note that the codimension of the boundary of the Satake-Baily-Borel compactification 
$\partial \overline{X}^{\rm SBB}:=
\overline{X}^{\rm SBB}\setminus X$ is at least $2$, following 
from our assumption that $\G$ is not isogenous to ${\rm SL}(2)$. Indeed, 
for such $G$, any maximal real parabolic subgroup $P$ has 
unipotent radical of dimension at least $2$ so that  
Levi part of $P$ has real codimension at least $3$. 
The existence of the special reflective modular form 
implies 
\begin{align}\label{br.L}
\sum_{i}\frac{d_{i}-1}{d_{i}}B_{i} \sim_{\Q} s(X)L. 
\end{align}
If we regard the holomorphic section satisfying 
Assumption \ref{Ass1} \eqref{item1.1} as a section of the ample line bundle $L^{\otimes s(X) N}$, 
it follows that the complement of the vanishing locus is affine 
but that is nothing but $\overline{X}^{\rm SBB}\setminus 
\cup_i \overline{B_{i}}^{\rm SBB}$ which includes 
$X^{o}$. This proof reflects the idea of \cite{Borcherds.Enriques}. 

From \eqref{br.L} and the definition of $L$ it follows that 
\begin{align}\label{aut.K}
-K_{\overline{X}^{\rm SBB}}\sim_{\Q}(s(X)-1)L
\end{align}
in ${\rm Pic}(\overline{X}^{\rm SBB})\otimes \Q$. 
Hence,  
$-K_{\overline{X}^{\rm SBB}}$ is ample 
$\Q$-Cartier 
if $s(X)>1$. Similarly, 
$K_{\overline{X}^{\rm SBB}}$ is ample 
$\Q$-Cartier 
(resp., $K_{\overline{X}^{\rm SBB}}=0$) 
if $s(X)<1$ (resp., if $s(X)=1$). 
On the other hand, 
from \cite[3.4, 4.2 (also see 1.3)]{Mum77}, 
$\overline{X}^{\rm SBB}$ 
is obtained as a projective spectrum of a 
certain log canonical ring, hence the pair 
$(\overline{X}^{\rm SBB},\sum_{i}\frac{d_{i}-1}{d_{i}}
\overline{B_{i}}^{\rm SBB})$ 
has only  log canonical  singularity (as a pair) and 
$K_{\overline{X}^{\rm SBB}}+\sum_{i}\frac{d_{i}-1}{d_{i}}
\overline{B_{i}}^{\rm SBB}$ is ample 
(see also \cite[3.4, 3.5]{Alexeev}). 
Thus $\sum_{i}\frac{d_{i}-1}{d_{i}}
\overline{B_{i}}^{\rm SBB}$ is also $\Q$-Cartier 
so that $X$ itself is also log canonical. 

On the other hand, recall that the construction of the 
Baily-Borel compactification \cite{BB} is a projective spectrum of the graded ring of automorphic forms and 
$L$ is the $c$ multiple tensors of its tautological line bundle $\mathcal{O}(1)$ in the construction. Hence, it is ample so that our latter statements of the above theorem all follow from 
\eqref{aut.K}. This fact is more clarified in 
\cite[\S 3, \S 4]{Mum77}. We complete the proof. 
\end{proof}
\begin{Rem}
The above results are analogous to the 
Fanoness results in 
\cite{DN}, 
(resp., \cite[\S 2]{Huy}
also \cite[\S 4]{JLi}) in the context of moduli 
 of (semi)stable bundles over 
 curves (resp., surfaces). 
 For the case over surfaces, 
 the determinant line bundle which descends to the Donaldson-Uhlenbeck compactification is used in the place of automorphic line bundle $L$.  
\end{Rem}
\begin{Rem}
Case \eqref{lcmodel} is a variant of the so-called 
``low weight cusp form trick" (cf., e.g., \cite{GHS}). 
See also \cite{Gr}, \cite[\S 5.5]{Gri} and references therein. 
\end{Rem}
We introduce the following notion. 
\begin{Def}\label{naked}
We call a cusp $F$ of $\overline{X}^{\rm SBB}$ 
 {\it naked} if 
it is not contained in 
${\rm Supp}(\overline{B_{i}}^{\rm SBB})\cap \partial 
\overline{X}^{\rm SBB}$ for any $i$. 
Further, we call it 
{\it minimal naked} 
if it is minimal 
with respect to the closure relation 
among naked cusps, i.e., 
$\overline{F}\setminus F$ 
is contained in 
$(\cup_i 
{\rm Supp}(\overline{B_{i}}^{\rm SBB}))\cap \partial 
\overline{X}^{\rm SBB}$. 
Also, we call 
$\partial \overline{X}^{\rm SBB}\setminus \bigcup_{i}
\overline{B_{i}}^{\rm SBB}$ {\it the naked locus}. 
\end{Def}

Below, we observe a certain weakening of connected-ness of 
cusps closure in the case of $s(X)>1$, i.e., Fano case. 
This follows from 
\cite[4.4, 6.6 (ii)]{Amb}, 
\cite[8.1]{Fjn}, 
\cite[\S 3]{Fjn.Fano}, 
\cite[1.2]{FG} 
as the proof 
below, which is essentially 
just a review to make 
our logic more self-contained. 
Compare with our 
examples of the modular varieties 
given in the next section. 

\begin{Cor}[Boundary structure for Fano Shimura varieties]
\label{connected}
Let us assume 
the same assumption of Theorem \ref{Fanoness} and 
further that $s(X)>1$. Then, the naked locus 
$$\partial \overline{X}^{\rm SBB}\setminus \bigcup_{i}
\overline{B_{i}}^{\rm SBB}$$ is connected and its closure is nothing but the 
non-log-terminal locus of 
$\overline{X}^{\rm SBB}$. 
More strongly, 
there is at most one minimal 
naked cusp with respect to the 
closure relation. 

Furthermore, 
if we suppose such a minimal naked cusp $F$ exists, there is an  
effective $\Q$-divisor $D_F$ 
such that $(\overline{F},D_F)$ 
has only klt singularities and 
is a log Fano pair,  
i.e., $-K_F-D_F$ is ample and 
$\Q$-Cartier. 
For instance, if $F$ is a 
modular curve, it is rational i.e., 
$\overline{F}\simeq \mathbb{P}^1$ 
(with ``Hauptmodul"). 
\end{Cor}

\begin{proof}
Firstly, we prepare the following general lemma 
(compare with e.g., \cite[\S 3]{Alexeev}). 

\begin{Lem}[Log canonical centers]\label{lc.center}
\begin{enumerate}
    \item \label{lc1}
Under the notation of \S \ref{Notation} 
for general Shimura varieties, 
without the above assumptions in Corollary \ref{connected},
 the log canonical 
centers of 
$(\overline{X}^{\rm SBB},\sum_{i}\frac{d_{i}-1}{d_{i}}\overline{B_{i}}
^{\rm SBB})$ 
are nothing but cusps of the 
Satake-Baily-Borel compactification 
$\overline{X}^{\rm SBB}$. 
\item \label{lc2}
Under the above assumptions in Corollary \ref{connected}, 
the log canonical 
centers of 
$\overline{X}^{\rm SBB}$ 
are nothing but 
cusps of the 
Satake-Baily-Borel compactification 
$\overline{X}^{\rm SBB}$ 
which are not contained in $\cup_{i}{\rm Supp}
(\overline{B_{i}}^{\rm SBB})$. 
\end{enumerate}
\end{Lem}
\begin{proof}[proof of Lemma \ref{lc.center}]
As in \cite[Chapter III, \S 7]{AMRT}, 
we replace the (implicit dividing) 
discrete group $\Gamma$ in the notation 
\S \ref{Notation} by its neat subgroup (cf., \cite{AMRT}) 
of finite index. 
In that way, we replace $X$ (and $\overline{X}^{\rm SBB}$) 
by its finite cover so that 
the first desired claim 
\eqref{lc1} 
for the log canonical centers of 
$(\overline{X}^{\rm SBB},\sum_{i}\frac{d_{i}-1}{d_{i}}\overline{B_{i}}
^{\rm SBB})$ is reduced 
to the case when there is no $B_i$. 

Then, there is a log resolution of 
$(\overline{X}^{\rm SBB},\sum_{i}\frac{d_{i}-1}{d_{i}}\overline{B_{i}}
^{\rm SBB})$ as a 
toroidal compactificaftion \cite[chapter III]{AMRT}, 
see especially {\it loc.cit} 
6.2. By its construction in {\it op.cit} of toroidal 
nature (see again e.g., \cite[\S 3]{Alexeev}), 
all the exceptional prime divisors have 
the discrepancy $-1$ and hence the claim \eqref{lc1} for 
the log canonical centers of 
$(\overline{X}^{\rm SBB},\sum_{i}\frac{d_{i}-1}{d_{i}}\overline{B_{i}}
^{\rm SBB})$ follows. 

For the proof of 
latter claim \eqref{lc2}, 
note that the existence of special reflective 
modular form implies 
$\sum_{i}\frac{d_{i}-1}{d_{i}}\overline{B_{i}}
^{\rm SBB}$ is a $\mathbb{Q}$-Cartier divisor 
by \eqref{br.L} of 
the proof of Theorem \ref{Fanoness}. 
Hence, the 
note that 
log canonical centers of 
$\overline{X}^{\rm SBB}$
form a subset of 
the lc centers of \eqref{lc1} which are not 
contained in the support of the effective  
$\mathbb{Q}$-Cartier divisor 
$\sum_{i}\frac{d_{i}-1}{d_{i}}\overline{B_{i}}
^{\rm SBB}$. Hence, the claim of Lemma \ref{lc.center} \eqref{lc2}. 
\end{proof}
Now we start the proof of Corollary \ref{connected}. 
We take the union of the minimal 
naked cusps of $\overline{X}^{\rm SBB}$ 
as $W$ and 
put the reduced scheme structure on it. 
We denote the corresponding coherent 
ideal sheaf of 
$\mathcal{O}_{\overline{X}^{\rm SBB}}$ 
as $I_{W}$. 

From a vanishing theorem of  \cite[4.4]{Amb},\cite[8.1]{Fjn}, 
whose absolute  non-log version is enough for our particular purpose here, we have 
$H^{1}(\overline{X}^{\rm SBB},I_{W})=0$. 
On the other hand, 
$H^{0}(\overline{X}^{\rm SBB},I_{W})=0$ 
also holds since 
it is a linear subspace of 
$H^{0}(\overline{X}^{\rm SBB},\mathcal{O})$ 
which is identified with $\mathbb{C}$ because of the 
properness of $\overline{X}^{\rm SBB}$, 
combined with the fact that $W\neq \emptyset$. 
Hence, combined with standard cohomology exact sequence arguments, 
$H^{0}(\mathcal{O}_{W})\simeq \C$ follows. 
Hence, it 
implies the connectivity of $W$, 
so that there is at most $1$ minimal naked cusp $F$. 

For such $F$, 
the existence of $D_F$ on 
the closure $\overline{F}$ 
follows 
from applying the log 
canonical subadjunction 
\cite[1.2]{FG} to 
$F\subset (\overline{X}^{\rm SBB},0)$. 
\end{proof}

We make a caution that the 
above Corollary \ref{connected} 
does not claim the naked cusp always 
has log terminal singularity. 
Nevertheless, 
in the $\Q$-rank $1$ case, we have the 
following. 

\begin{Cor}[$\Q$-rank $1$ case]\label{Qrk1}
Under the same assumptions of Theorem \ref{Fanoness} 
with $>1$, 
if further $\Q$-rank of $\G$ is $1$ (e.g., 
when $G\simeq U(1,n)$ for some 
$n$ so that 
$G/K$ is an $n$-dimensional 
complex unit ball), 
only either one of the  followings hold.
\begin{enumerate}
\item There is exactly one 
naked cusp $F$ of $\overline{X}^{\rm SBB}$ 
which is an isolated non-log-terminal  locus 
but at worst log canonical. 
Furthermore, there is an effective 
$\Q$-divisor $D_F$ such that $(F,D_F)$ 
is a klt log Fano pair hence in 
particular, the modular branch divisor in $F$ 
is nonzero effective. 
\item \label{rat.con} No naked cusp exists and $X$ is rationally connected, 
i.e., two general points are connected by a 
rational curve 
and has at worst log terminal singularities. 
Furthermore, $X\setminus {\rm Supp} \cup_i B_i$ is affine (not only quasi-affine). 
\end{enumerate}
\end{Cor}
\begin{proof}
Note that the condition 
that $\Q$-rank of $\G$ is $1$ implies that 
the boundary strata of the Satake-Baily-Borel 
compactification of $X$ are all compact and 
do not have closure relations. Thus, 
among the above statements, 
the only assertion 
which does not follow trivially from Corollary \ref{connected} 
is 
the rationally connected assertion for the latter case 
\eqref{rat.con}. 
We confirm it as follows: the non-existence of naked cusp means $\overline{X}^{\rm SBB}
\setminus X$ is included in 
$\cup_{i}{\rm Supp}(\overline{B_{i}}^{\rm SBB})$ 
which implies the log terminality of $X$. Hence, 
it is rationally connected by a theorem of Zhang  \cite{Zhang}. Finally, 
$X\setminus {\rm Supp} \cup_i B_i$ is affine 
by the 
proof of Theorem \ref{Fanoness} and the assumption that 
there are no naked cusps. 
\end{proof}

Here is a version of the converse direction of Theorem \ref{Fanoness}. 

\begin{Thm}[Abstract existence of 
special modular forms]
\label{converse.Fanoness}
We follow the notation of Theorem \ref{Fanoness}. 
If $\overline{X}^{\rm SBB}$ satisfies either 
\begin{itemize}
\item 
$K_{\overline{X}^{\rm SBB}} 
\equiv 0$ or 
\item either $K_{\overline{X}^{\rm SBB}}$ or $-K_{\overline{X}^{\rm SBB}}$ is ample with 
Picard number $1$, 
\end{itemize}
then there are special reflective modular forms satisfying  
Assumption \ref{Ass1} \eqref{item1.1} for some $s(X)\in \Q_{>0}$ 
and sufficiently divisible $N\in \mathbb{Z}_{>0}$. 
Furthermore, if it is of a certain  orthogonal type, i.e., $\mathbb{G}$ is isogenous to $SO(\Lambda)$ for  $\Lambda=U\oplus U(l)\oplus N$ with some negative definite lattice $N$ and $l\in \mathbb{Z}_{>0}$,  the modular forms are necessarily Borcherds lift of 
some nearly holomorphic elliptic 
$Mp_2(\mathbb{Z})$-modular forms of a 
specific  principal part of the Fourier expansion in the sense of 
\cite{Borcherds.general}, 
\cite[\S 1.3, \S 3.4]{Bruinier.book}. 
\end{Thm}

\begin{proof}
Given the proof of Theorem \ref{Fanoness}, we can 
almost trace back the 
arguments as follows. In either cases, the automorphic line bundle $L$ is 
proportional to $K_{\overline{X}^{\rm SBB}}$ in 
${\rm Pic}(\overline{X}^{\rm SBB})$, 
 hence so is it to 
$\sum_{i}\frac{d_{i}-1}{d_{i}}
\overline{B_{i}}^{\rm SBB}$. Therefore, 
$\mathcal{O}(N(s(X)L-\sum_{i}\frac{d_{i}-1}{d_{i}}
\overline{B_{i}}^{\rm SBB}))$ is trivial for some $s(X), N$. 
The last assertion follows from 
\cite[5.12]{Bruinier.book}, \cite[1.2]{Bruinier.article}.
\end{proof}


\subsection{Modular varieties with big anti-canonical classes}
\label{-Kbig.subsec}

Recall that Gritsenko-Hulek \cite{GrHu} (resp., 
Maeda \cite{Maeda}) discuss  the classes of 
reflective orthogonal modular forms 
(resp., unitary modular forms) 
satisfying Assumption \ref{Ass1} \eqref{item1.2} 
with $s(X)>1$ 
and proved uniruledness of  $X$ and constructs some 
examples. 

This subsection proves the following a  
slight refinement of their results, 
which applies to the examples constructed in {\it loc.cit}. 

\begin{Thm}[{cf., \cite[2.1]{GrHu}, \cite[4.1]{Maeda}}]\label{-K.big}
We follow the notation of \S \ref{Notation}, 
and discuss Shimura varieties $X=\Gamma\backslash D$ 
for a priori general $G$. 
If there is a reflective modular form $\Phi$ which satisfies 
Assumption \ref{Ass1} \eqref{item1.2} with some $s(X)\in \mathbb{Q}_{>1}$, 
, we define 
$V_{\Phi}:=\cup_{F}\overline{F} \subset \partial 
\overline{X}^{\rm SBB}$ 
where $F$ runs through all cusps along which 
$\Phi$ does not vanish (as a function, 
or a section of $L^{\otimes s(X)N}$). 
Then, the following holds. 
\begin{enumerate}
    \item 
The Satake-Baily-Borel compactification $\overline{X}^{\rm SBB}$ of 
$X=\Gamma\backslash D$ only has log canonical singularities, 
$X^{o}$ is quasi-affine and 
$-K_{\overline{X}^{\rm SBB}}$ is big. 

    \item 
For any two closed points 
$x, y \in \overline{X}^{\rm SBB}$, there are 
union of rational curves $C$ such that 
$C\cup V_{\Phi}$ is connected 
(i.e., rationally chain connected modulo $V_{\Phi}$ cf., 
\cite[1.1]{HM}). 
In particular, $X$ is uniruled. If $G=U(1,n)$ for some $n$, 
then $\overline{X}^{\rm SBB}$ is 
even rationally chain connected. 
\item If we consider the set of 
cusps outside $V_{\rm \Phi}$, there is at most 
$1$ minimal element (cusp) with respect to the 
closure relation. 
\end{enumerate}
\end{Thm}

\begin{proof}
We first consider (i) of the above theorem. 
From the existence of $\Phi$, it follows in the same way that 
$$
-K_{\overline{X}^{\rm SBB}}\sim_{\Q}
(s(X)-1)L+\sum_i 
\biggl(\frac{d_{i}-1}{d_i}-c_i\biggr) 
\overline{B_i}^{\rm SBB}, 
$$
hence it is big. The proofs of the other assertions in (i) are the 
same as those of Theorem \ref{Fanoness}. 
For (ii), note that the non-klt locus of 
$(\overline{X}^{\rm SBB},\sum_i (\frac{d_{i}-1}{d_i}-c_i) \overline{B_i}^{\rm SBB})$ 
is the union of log canonical centers of 
$(\overline{X}^{\rm SBB},
\sum_i \frac{d_{i}-1}{d_i} \overline{B_i}^{\rm SBB})$ 
which are not inside ${\rm Supp}({\rm div}(\Phi))$. 
Hence, the assertion (ii) 
directly follows from 
\cite[1.2]{HM} for 
$(\overline{X}^{\rm SBB},
\sum_i \frac{d_{i}-1}{d_i} \overline{B_i}^{\rm SBB})$. The assertion for the unitary case holds 
since the cusps are all  $0$-dimensional 
(cf., e.g., \cite[section 4]{Behrens}). 
Indeed, it follows since the Levi part of 
real parabolic subgroup 
of $G$ corresponding to the cusps are $U(0,n-1)$, 
which is trivial. 
For (iii), the same 
arguments as Corollary \ref{connected}, similarly applying  
\cite[4.4, 6.6(ii)]{Amb} or \cite[8.1]{Fjn} 
to the log canonical Fano pair 
$(\overline{X}^{\rm SBB},
\sum_i (\frac{d_{i}-1}{d_i}-c_i) \overline{B_i}^{\rm 
SBB})$, give a proof. 
\end{proof}

\begin{Rem}
We can also show a variant of Corollary  \ref{connected}, Theorem \ref{-K.big} (iii)  under 
general 
{\it meromorphic} modular forms 
if we replace the use of 
\cite[6.6(ii)]{Amb} by 
\cite[4.4]{Amb} or \cite[6.1.2]{Fjn.MSJ}. 
However, because the obtained statement is rather complicated 
and no interesting applications have been found (yet at least), 
we omit it in this paper. 
\end{Rem}


\bigskip
We conclude this section by posing a natural problem. 
\begin{Prob}\label{novanishing}
In specific situations,  
e.g., 
when $\mathbb{G}=SO(\Lambda\otimes \Q)$ for a quadratic 
lattice $\Lambda$, or in the unitary 
modular case corresponding to 
a Hermitian lattice as later 
subsection \S 
\ref{prepare.Hlattice}, 
the assertions of 
Corollaries \ref{connected}, 
\ref{Qrk1}, Theorem \ref{-K.big} (iii) 
can be phrased 
in a purely lattice theoretic 
manner. Is there a more lattice theoretic or number theoretic 
proof without the use of a  
vanishing theorem in algebraic geometry? 
\end{Prob}

\vspace{3mm}

\section{Examples of Fano and K-ample cases}\label{Ex.sec} 

We provide examples of which 
Theorems \ref{Fanoness}, Corollary \ref{connected}, Corollary 
\ref{Qrk1},  Theorem \ref{converse.Fanoness} 
in \S \ref{Fano.subsec} apply. 
In the examples, 
the compactified modular varieties are either Fano varieties 
or with ample canonical classes. 
There are also some examples with 
$s(X)=1$, for instance \cite{FSM} 
(cf., also earlier \cite{BN} 
with a weaker statement) 
but we do not focus such cases 
in this paper. 

\subsection{Siegel modular cases}

We start by discussing 
the Satake-Baily-Borel compactifications of 
some semi-classical 
modular varieties, which we show to fit our picture. 
The examples in this subsection and 
the next subsection \ref{ortho.PartI} 
do not use explicit modular forms but they are 
Fano varieties so that the converse 
theorem \ref{converse.Fanoness} applies to imply the (abstract) 
existence of 
special reflective modular forms. 

The examples with explicit special reflective modular forms, to which we can apply Theorem \ref{Fanoness} 
will be discussed from 
the next section \S \ref{ortho.PartII}. 
Here are two examples of Siegel modular varieties whose Satake-Baily-Borel compactifications are Fano varieties. 

\begin{Ex}[\cite{Igusa}]
The Satake-Baily-Borel compactification of 
the moduli of principally polarized 
abelian surfaces 
$\overline{A_2}^{\rm SBB}$ is known to be a 
weighted projective hypersurface in 
$\mathbb{P}(4,6,10,12,35)$ 
of degree $70$ 
with the coarse moduli 
isomorphic to 
$\mathbb{P}(2,3,5,6)$ 
by relating to the 
invariants of genus $2$ curves, hence 
binary sextics. 
Note that the 
adjunction does not work 
due to non-well-formedness, 
as indeed one has 
non-trivial isotropy 
($\mu_2$) 
along a divisor in 
the moduli stack. 
The reduction of the natural 
Faltings-Chai model 
over $\mathbb{F}_{p}$ 
are also determined 
(cf., \cite{Ichikawa, vdG2}). 
\end{Ex}

\begin{Ex}[{cf., \cite[5.2]{vdG} (also \cite{Igusa})}]
The Satake-Baily-Borel compcatification of 
the moduli of principally polarized 
abelian surfaces with level $2$ structure 
$\overline{\Gamma(2)\backslash \mathfrak{H}}^{\rm SBB}$ is known to be a quartic 
$3$-fold 
\begin{equation}
\sum_{i=0}^{5}x_i
=(\sum_{i=0}^{5}x_i^2)^2
-4(\sum_{i=0}^{5}x_i^4)=0, 
\end{equation}
with 
non-isolated singularities 
along $15$ lines. Since this is a hypersurface, 
it is clearly Gorenstein and has ample anticanonical 
class. It also follows from \cite[\S 3, \S 4]{Mum77} (cf., also  \cite[3.5]{Alexeev}) again that it is at least log canonical. 
\end{Ex}

\subsection{Orthogonal modular cases, Part I}
\label{ortho.PartI}

Below, we consider the cases where $\mathbb{G}=SO(\Lambda\otimes \mathbb{Q})$  for a quadratic lattice $(\Lambda,(\ ,\ ))$ 
of signature $(2,n)$ with $n\in \mathbb{Z}_{>0}$. 
We realize the Hermitian symmetric domain 
$X=G/K$ as 
$G/K\simeq \D_\Lambda$ which is defined as one of (the isomorphic two) 
connected components of 
\[\{v\in\mathbb{P}(\Lambda\otimes\C)\mid (v,v)=0,\  (v,\overline{v})>0\}.\] 

We keep this notation throughout 
in the discussion of orthogonal 
modular varieties. 
Our first two examples in this Part I are understood via  moduli-theoretic methods and GIT as follows. 

\begin{Ex}[Hilbert]
The GIT compactification of the moduli of cubic surfaces (\cite[\S 4.2]{OSS}) 
is known to be isomorphic to the 
Satake-Baily-Borel compactification 
of the stable locus which admits uniformization of complex ball 
(cf., \cite{ACT}). 
Hilbert's invariant calculation 
in his thesis tells this is 
$\mathbb{P}(1,2,3,4,5)$, hence the only cusp is not naked 
because of the log terminality. 
Obviously, it is also a 
($\Q$-)Fano variety. 
This is also one of the simplest 
examples of the K-moduli variety 
of Fano varieties (\cite[\S 4.2]{OSS}). 
\end{Ex}

Given \cite{Ma.gentype}, 
it is reasonable 
to ask the following 
problem in general. 
\begin{Prob}
Classify the lattices 
$\Lambda$ of signature $(2,n)$
such that 
the Satake-Baily-Borel 
compactification 
$\Gamma \backslash 
\mathcal{D}_{\Lambda}$ 
are Fano varieties, 
especially when 
$\Gamma=O^+(\Lambda)$ or 
$\widetilde{O}^+(\Lambda)$. 
\end{Prob}

From what follows, our arithmetic subgroup 
satisfies $\Gamma$ is either 
$O^+(\Lambda)$ or the stable 
orthogonal group $\widetilde{O}^+(\Lambda)$. 

\vspace{3mm}
\begin{Ex}[Moduli of elliptic K3 surfaces]
We consider the moduli $M_{W}$ of Weierstrass elliptic K3 surfaces, 
which is an open subset of 
$O^{+}(\Lambda)\backslash \mathcal{D}_{\Lambda}$
for $\Lambda:=U^{\oplus 2}\oplus E_{8}(-1)^{\oplus 2}$. 
We consider its 
Satake-Baily-Borel compactification (\cite[Theorem 7.9]
{OO18}), 
which we denote $\overline{M_{W}}^{\rm SBB}$ here. 
Recall from {\it loc.cit} \S 7.1 
that there are exactly two $1$-cusps 
intersecting at the only $0$-cusp. 
Two $1$-cusps are 
$M_{W}^{\rm nn}$ with canonical Gorenstein singularity and $M_{W}^{\rm seg}$ with toroidal singularity (including the 
$0$-cusp 
$\overline{M_{W}^{\rm nn}}\cap \overline{M_{W}^{\rm seg}}$)  hence 
$\overline{M_{W}}^{\rm SBB}$  also only has log terminal singularity (\cite[Part I, \S 2]{Od20}). 
The notation of our superscripts ``nn" and ``seg" 
follow that of \cite[Chapter 7]{OO18} where some 
collapsing of hyperK\"ahler metrics to 
{\it segment} i.e., $[0,1]$ 
is partially observed along $M_W^{\rm seg}$, 
and also that {\it non-normal} degenerations are parametrized 
by $M_W^{\rm nn}$. 

We recall that 
$\overline{M_W}^{\rm SBB}$ coincides with a 
certain  GIT quotient of a weighted projective space 
(\cite[Theorem 7.9]{OO18}). Using the fact as well as 
some analysis of singularities along the 
$1$-cusps in \cite[Part I]{Od20}, we prove the 
following. 

\begin{Thm}\label{ellK3.moduli}
$\overline{M_W}^{\rm SBB}$ is a $18$-dimensional 
log terminal rational Fano variety of 
Picard rank $1$, 
although not isomorphic to any 
weighted projective space. 
Its two $1$-cusps 
$M_W^{\rm seg}$ and $M_W^{\rm nn}$ 
are both non-naked. 
\end{Thm}
\begin{proof}
The description of $\overline{M_W}^{\rm SBB}$ 
as a GIT quotient 
\cite[Theorem 7.9]{OO18} allows us 
to apply \cite[Corollary 3]{BGLM} to confirm 
there is an effective $\Q$-divisor $D$ on 
$\overline{M_W}^{\rm SBB}$ such that 
$-K_{\overline{M_W}^{\rm SBB}}-D$ is ample. 
Therefore, $-K_{\overline{M_W}^{\rm SBB}}$ is big. 
On the other hand, $\overline{M_{W}}^{\rm SBB}$ 
has Picard rank $1$ because of the same 
GIT quotient description. 
Hence, the bigness of 
$-K_{\overline{M_W}^{\rm SBB}}$ implies it is actually 
even ample i.e., $\overline{M_W}^{\rm SBB}$ is a 
Fano variety. 

The fact that both $1$-cusps are 
 non-naked are follows from Corollary \ref{connected}, 
 because $\overline{M_W}^{\rm SBB}$ is 
 log terminal as proven in \cite[Part I, \S 2]{Od20}. (The log terminality also 
 follows from \cite[Theorem1]{BGLM} combined again 
 with the fact that 
 $\overline{M_{W}}^{\rm SBB}$ has Picard rank $1$.) 
 As for the rationality of  $M_{W}$, 
 \cite{Lej} proved it, 
 based on more classical 
 rationality result of the moduli space of 
 hyperelliptic curves (of genus $5$). 
 
 The only remained thing to prove in the above theorem is 
 that  $\overline{M_{W}}^{\rm SBB}$ is not a weighted projective space. 
 From the analysis of 
singularity type along 
$1$-cusp $M_W^{\rm nn}$ in \cite[Part I, Theorem 2.2]{Od20}, it easily follows that the 
local fundamental group 
along the transversal slice 
is $(\mathbb{Z}/2\mathbb{Z})^{4}$ 
hence not cyclic. In 
particular, 
$\overline{M_W}^{\rm SBB}$ 
can not be a weighted 
projective space. We complete the proof of 
Theorem \ref{ellK3.moduli}. 
\end{proof}

As a corollary, we also observe 
the following. 

\begin{Cor}
On the orthogonal modular variety 
$\overline{M_{W}}^{\rm SBB}$, 
there are special reflective modular forms which satisfy    Assumption \ref{Ass2} \eqref{item1.1} (of \S \ref{special.form.section}) 
for some $s(X)>1$ 
and sufficiently divisible $N\in \mathbb{Z}_{>0}$. 
\end{Cor}
\begin{proof}
By the above theorem \ref{ellK3.moduli}, 
we can apply Theorem \ref{converse.Fanoness} 
to complete the proof. 
\end{proof}

\end{Ex}


\subsection{Orthogonal modular cases, Part II}
\label{ortho.PartII}

From here, we use the Borcherds products  
to show that various Satake-Baily-Borel 
compactifications of orthogonal modular 
varieties are Fano varieties or 
log canonical models. 

\subsubsection*{Notation}
Let
\[\mathcal{H}(\ell):=\{v\in \D_\Lambda\mid (v,\ell)=0\}\]
be the special divisor with respect to $\ell\in\Lambda$ with $(\ell,\ell)<0$.
For any primitive element $r\in \Lambda$ satisfying $(r,r)<0$, we define the {\em reflection}  $\sigma_r\in O^+(\Lambda)(\Q)$ with respect to $r$ as follows:
\[\sigma_r(\ell):=\ell -\frac{2( \ell,r)}{(r,r)}r.\]

Then, the union of ramification divisors of $\pi_{\Gamma}\colon \D_\Lambda\to\Gamma\backslash \D_\Lambda$  is 
  \[\bigcup_{\substack{r\in \Lambda/\pm\mathrm{:primitive}\\ \sigma_r\in\Gamma \ \mathrm{or}\  -\sigma_r\in\Gamma}}\mathcal{H}(r)\]
  by \cite{GHS} for $\Gamma\subset 
O^+(\Lambda)$ and $n>2$. 
They also showed that the ramification degrees are $2$.
We sometimes denote 
  $\pi_{\Gamma}$ as $\pi$. 
We also define
 \begin{align*}
    \mathcal{H}_{-2}&:=\bigcup_{\ell\in \Lambda,\ \ell^2=-2}\mathcal{H}(\ell)\\
    \mathcal{H}_{-4}&:=\bigcup_{\ell\in \Lambda,\ \ell^2=-4}\mathcal{H}(\ell)\\
    \mathcal{H}_{-4,special\mathchar`-even}&:=\bigcup_{\ell\in \Lambda:special\mathchar`-even,\ \ell^2=-4}\mathcal{H}(\ell).
\end{align*}
Here we say a vector $r\in\Lambda$ is special-even (also called even type e.g., 
in \cite{Kondo02}) if $(\ell.r)$ is even for any $\ell\in\Lambda$,  
i.e., ${\rm div}(r)$ is even integer, 
so that the corresponding 
reflection lies in $\Gamma$. 
We define ${\rm div}(r)$ is the positive generator of the ideal 
\[\{(\ell,r)\mid\ell\in\Lambda\}.\]

\begin{Rem}
\label{rem:orthog_a}
Below, for orthogonal cases, if $f$ is a modular form corresponding to a section satisfying Assumption \ref{Ass2} (i), we can compute $s(X)=\frac{k}{2mn}$. 
Here, $k$ is the weight of $f$ and $m$ is the multiplicity of ${\rm div} f$, and $n={\rm dim}X$.
\end{Rem}

\vspace{3mm}

\begin{Ex}\label{Unimodular case I}

Let $II_{2,26}=U\oplus  U\oplus E_8(-1)\oplus E_8(-1)\oplus E_8(-1)$ be an even unimodular lattice of signature $(2,26)$. 
We consider the case 
$\Gamma=O^+({\Lambda})$. 
There is the modular form  $\Phi_{12}$ of weight 12 on  $\D_{II_{2,26}}$ by Borcherds 
\cite{infinite} with 
\begin{align}\label{Phi12.restrict}
{\rm div}{\Phi_{12}}=\mathcal{H}_{-2}.
\end{align}
On the other hand, the ramification divisors of the map $\pi\colon II_{2,26}\to X:=O^+(II_{2,26})\backslash \D_{II_{2,26}}$ are  
$\mathcal{H}_{-2}$ by the even unimodularity of $\Lambda$ and \cite{GHS}. 

Now $\Phi_{12}^{2\times 26}$ 
satisfies Assumption 1.2 (i) 
with $s(X)=\frac{3}{13}$ 
and by Theorem \ref{Fanoness} (iii) 
so that  
the Satake-Baily-Borel compactification 
$\overline{X}^{\mathrm{SBB}}$ 
of the $26$-dimensional 
orthogonal modular variety 
$X=O^+(II_{2,26})\backslash 
\mathcal{D}_{II_{2,26}}$ 
is a log canonical model i.e., 
with ample canonical divisor 
$K_{\overline{X}^{\rm SBB}}$ 
and at worst log canonical 
singularities. Let us specify 
and study the 
non-log-terminal locus or the log canonical center. 

First, recall that there are exactly $24$ $1$-cusps, which 
correspond to Niemeier lattices 
and all intersect 
at a common closed point 
(cf., e.g., \cite[1.1]{GB}). 
In particular, there is a $1$-cusp 
which is 
the compactification of 
the modular curve 
$SL(2,\Z)\backslash \mathbb{H}$ 
corresponding to 
the Leech lattice. We denote the particular $1$-cusp as 
$C_{\rm Leech}$. 

For the Harish-Chandra-Borel embedding 
$$\mathcal{D}_{II_{2,26}}\subset 
\mathcal{D}_{II_{2,26}}^{c}
\subset 
\mathbb{P}(II_{2,26}\otimes \mathbb{C}),$$ 
$\mathcal{O}_{\mathbb{P}(II_{2,26}\otimes \mathbb{C})}(1)$ 
restricts to 
$\mathcal{O}_{\mathbb{P}^1}(1)|_{\mathbb{H}}$ for any 
$1$-cusp $\mathbb{H}\subset \mathbb{P}^1$. 
For instance, by \cite[\S 10]{infinite}, \cite[1.2]{GB}, 
$\Phi_{12}$ restricts 
to the Ramanujan cusp form 
$\Delta_{12}(q):=q\prod_{n\ge 1} (1-q^n)^{24}$ 
of weight $12$ on $C_{\rm Leech}$. 
Since the only modular branch divisor 
is $\mathcal{H}_{-2}$, together with \eqref{Phi12.restrict} and 
Lemma \ref{lc.center}, it implies that 
the only log canonical center is the $C_{\rm Leech}$. 
Recall that through the well-known isomorphism 
$SL(2,\Z)\backslash \mathbb{H} 
\simeq \mathbb{A}^1(\mathbb{C})
\subset \mathbb{P}^1(\mathbb{C})$, 
the elliptic modular forms of 
weight $12 k$ can be regarded with a section of 
$\mathcal{O}_{\mathbb{P}^1}(k)$, 
at the level of coarse moduli. 
In other words, $\mathcal{O}_{\mathbb{P}^1}(12 k)|_
{\mathbb{H}}$ descends to a line bundle 
$\mathcal{O}_{\mathbb{P}^1}(k)$ 
on $\mathbb{P}^1\simeq SL(2,\Z)\backslash \overline{\mathbb{H}}$ where 
$\overline{\mathbb{H}}$ denotes the rational closure of 
$\mathbb{H}$. 

In particular, 
$(2s(X)L.C_{\rm Leech})=1$, where $L$ follows the notation of 
\S \ref{Notation}. Equivalently 
$(K_{\overline{X}^{\rm SBB}}.C_{\rm Leech})=\frac{5}{3},$ 
$(\overline{B}.C_{\rm Leech})=1$ as $s(X)=\frac{3}{13}$. 
We summarize our conclusion in  
this case neatly as 
$II_{2,26}$ attracts special attention. 

\begin{Cor}[$II_{2,26}$ case]
The Satake-Baily-Borel compactification 
$\overline{X}^{\mathrm{SBB}}$ 
of the $26$-dimensional 
orthogonal modular variety 
$X=O^+(II_{2,26})\backslash 
\mathcal{D}_{II_{2,26}}$ 
is a log canonical model i.e., 
with ample canonical divisor 
$K_{\overline{X}^{\rm SBB}}$ 
and at worst log canonical 
singularities. Further, 
the non-log-terminal locus is the 
single $C_{\rm Leech}\simeq \mathbb{P}^1$ 
in the boundary $\partial X^{\rm SBB}$ 
which compactifies 
$1$-cusp 
$SL(2,\Z)\backslash \mathbb{H}$ 
and is characterized by that 
the corresponding isotropic plane 
$p\subset II_{2,26}\otimes \R$ 
satisfies that 
$(p^{\perp}\cap II_{2,26})/(p\cap II_{2,26})$ is the 
Leech lattice i.e., 
contains no roots. 
Its degree is 
$(K_{\overline{X}^{\rm SBB}}.C_{\rm Leech})=\frac{5}{3}.$ 
(resp., $(\overline{B}.C_{\rm Leech})=1$). 
\end{Cor}
Later in Example \ref{Unimodular case I_unitary}, we also construct a 
$13$-dimensional 
unitary modular subvariety 
which also compactifies with 
ample canonical class as the 
Satake-Baily-Borel compactification. 
\end{Ex}

\vspace{3mm}
\begin{Ex}
\label{Unimodular case II}
Let $\Lambda:=U\oplus U\oplus E_8(-1)$ be an even unimodular lattice 
of signature $(2,10)$. 
We again consider the case 
$\Gamma=O^+({\Lambda})$. 
Borcherds constructed a reflective modular form on $\D_{\Lambda}$.
\begin{Thm}[{\cite[10.1, 16.1]{infinite}}]
There is a reflective modular form $\Phi_{252}$ of weight 252 on 
$\D_{\Lambda}$ such that 
\[{\rm div}{\Phi_{252}}=\mathcal{H}_{-2}.\]
\end{Thm}
Here, by the map $\pi\colon  \D_{\Lambda}\to X:= O^+({\Lambda})\backslash\D_{\Lambda}$,  
the divisors $\mathcal{H}_{-2}$ maps to the unique 
branch divisors 
(cf., \cite[\S 2]{GHS}). 
Hence $\Phi_{252}^{10t}$ satisfies Assumption \ref{Ass2} (i) 
with $s(X)=\frac{63}{5}$ for some $t\in\Z$, 
and by Theorem \ref{Fanoness} (i), the compactified Shimura variety  $\overline{X}^{\mathrm{SBB}}$ is 
a Fano variety. 
Actually, 
\cite[1.1]{HashimotoUeda}, \cite[4.1]{DKW} (also 
attributed to H.Shiga and \cite{Looijenga}) 
shows it is the 
weighted projective space 
$\mathbb{P}(2,5,6,8,9,11,12,14,15,18,21)$. 

\end{Ex}

\vspace{3mm}
\begin{Ex}[Moduli of Enriques surfaces]
\label{Enriques}
The well-studied moduli space $M_{Enr}$ of (unpolarized) Enriques 
surfaces (cf., e.g., 
\cite{Namikawa, Sterk, Borcherds.Enriques, 
Kondo.Enriques}) 
also fit into our setting. 
Let $\Lambda_{Enr}:=  U \oplus U (2)\oplus  E_8(-2)$ be an even  lattice 
of signature $(2,10)$. 
Then the Shimura variety
\[M_{Enr}:=  O^+(\Lambda_{Enr})\backslash \D_{\Lambda_{Enr}}\]
is a 10-dimensional quasi-projective variety.
Now we review the ramification divisors of the natural map $\pi : \D_{\Lambda_{\mathrm{Enr}}}\to M_{\mathrm{Enr}}$ and moduli discription.
From \cite{GHS} and \cite{GH}, the ramification divisors are
\[\mathcal{H}_{-2}\cup\mathcal{H}_{-4,special\mathchar`-even}.\]
On the other hand, let 
\[\widetilde{M_{Enr}}:= \widetilde{O}^{+}(\Lambda_{Enr})\backslash \D_{L_{Enr}}\]
be a finite cover of $M_{\mathrm{Enr}}$.
Then the following are known. 
\begin{Prop}
\begin{enumerate}
    \item $M_{Enr}\backslash\pi(\mathcal{H}_{-2})$ is the so-called moduli space of Enriques surfaces (cf., e.g., 
    \cite{Namikawa}). 
    Moreover this is rational (Kondo \cite{Kondo.Enriques}). 
    \item $\widetilde{M_{Enr}}\backslash\pi(\mathcal{H}_{-2})$, which is a finite cover of $M_{\mathrm{Enr}}$, is the moduli space of Enriques surfaces with a certain level-2 structure.
    Moreover $\widetilde{M_{Enr}}$ and $\widetilde{M_{Enr}}\backslash\pi(\mathcal{H}_{-2})$ are of general type (Gritsenko-Hulek cf., \cite{GH}).
    \item $M_{Enr}\backslash(\pi(\mathcal{H}_{-2})\cup\pi(\mathcal{H}_{-4,special\mathchar`-even}))$ is the moduli space of non-nodal Enriques surfaces.
\end{enumerate}
\end{Prop}

Going back to our situation, 
we need special reflective modular forms satisfying Assumption \ref{Ass2} (i). 
Our input here is the following.

\begin{Lem}[{\cite{Borcherds.Enriques, 
Kondo02}}] 
There exist two reflective modular forms $\Phi_4$ and $\Phi_{124}$ on $\D_{L_{Enr}}$ of weights $4$, $124$  respectively such that;
\begin{align*}
    {\rm div}{\Phi_4}&=\mathcal{H}_{-2},\\
    {\rm div}{\Phi_{124}}&=\mathcal{H}_{-4,\textrm{special-even}}.
\end{align*}
\end{Lem}

We put $F_{128}:= \Phi_4\Phi_{124}$.
Then this is a weight $128$ modular form on $\D_{L_{Enr}}$ and 
${\rm div}(F_{128})$ is  exactly the ramification divisors of the map  $\pi:\D_{L_{Enr}}\to M_{Enr}$ with coefficients $1$. 
Now $F_{128}^2$ has a trivial character and satisfies Assumption \ref{Ass2} (i) with $s(X)=\frac{32}{5}$ and by Theorem \ref{Fanoness} (i), $\overline{M_{Enr}}^{\mathrm{SBB}}$ is a log canonical Fano variety. 

Actually, it is even log terminal 
without naked cusps as we confirm in the folowing.  
By \cite[3.3, 4.5]{Sterk}, 
there are only two 
$0$-cusps which correspond 
to an isotropic vector $e$ in 
the first summand $U$ 
and 
an isotropic vector $e'$ 
the second summand 
$U(2)$ of $\Lambda_{\rm Enr}$. 
They belong 
to the same $1$-cusp 
which corresponds to 
isotropic plane 
$\mathbb{Q}e\oplus 
\mathbb{Q}e'$. 
That $1$-cusp is 
contained in the closure of 
$\mathcal{H}_{-4,special\mathchar`-even}$ since $e$ and $e'$ 
are orthogonal to the 
(norm-doubled) root of 
$ E_8(-2)$, 
the third summand of 
$L_{\rm Enr}$. By 
{\it loc.cit}, 
the only other 
$1$-cusp corresponds to 
another isotropic plane 
$$p=\mathbb{Q}e'\oplus
\mathbb{Q}(2e+2f+\alpha)$$ 
where $e,f$ is the 
standard basis of the 
first summand $U $ 
and $\alpha$ is norm $-8$ 
integral vector in the 
third summand 
$ E_8(-2)$. 
Since $p$ is obviously 
orthogonal to the 
$-2$ vector $e-f\in 
U $, 
the corresponding 
$1$-cusp is also 
contained in the closure of 
the Coble locus $\mathcal{H}_{-2}$. 
Hence 
there are no naked cusps 
so that we conclude the following. 

\begin{Cor}
The 
Satake-Baily-Borel 
compactification $\overline{M_{Enr}}^{\mathrm{SBB}}$  of the 
moduli of Enriques surfaces 
$M_{Enr}$ 
is a log terminal 
Fano variety. 
\end{Cor}

\end{Ex}

\vspace{3mm}
\begin{Ex}[Moduli of log Enriques surfaces]
\label{log.Enriques}
For each $1\le k\le 10\ (k\neq 2)$, 
let $\Lambda_{\mathrm{logEnr},k}:= U (2)\oplus A_1\oplus A_1(-1)^{\oplus 9-k}$ be an even lattice of signature $(2,10-k)$. 
Then the associated Shimura variety $O^+(\Lambda_{logEnr,k})\backslash \D_{L_{logEnr}}$ is a 
(partial compactification of) the 
moduli space of log Enriques surface with 
$k$ $\frac{1}{4}(1,1)$ singularities. 
For the definition of log Enriques surfaces with $\frac{1}{4}(1,1)$ singularities, see \cite[Definition 2.1, 2.6]{DY}.
Yoshikawa \cite{Yoshikawa2} and Ma \cite{MaYos} constructed reflective modular forms on $\D_{L_{logEnr}}$ for $k\le 7$ 
which we use. 
\begin{Thm}[{\cite[Theorem 4.2(i)]{Yoshikawa2}}]
There is a reflective modular form $\Psi_4$ of weight $4+k$ on $\D_{\Lambda_
{logEnr,k}}$ with 
\[{\rm div}{\Psi_{4+k}}=\mathcal{H}_{-2}.\]
\end{Thm}
\begin{Thm}[{\cite[Appendix by Yoshikawa;  A.4, 
proof of A.5]{MaYos}}]
There is a reflective modular form 
$\Psi_{124,k}$ of weight 
$-k^2-9k+124$ 
on $\D_{\Lambda_{logEnr,k}}$ with 
\[{\rm div}{\Psi_{124,k}}=\mathcal{H}_{-4}.\]
\end{Thm}
On the other hand, the ramification divisors of the map $\pi\colon \D_{L_{logEnr,k}}\to X:= O^+(L_{logEnr,k})\backslash \D_{L_{logEnr,k}}$ is the union of special divisors with respect to $(-2)$-vectors and $(-4)$-vectors by the same discussion. 
As 
$(\Psi_{4+k}\Psi_{124,k})^{t(10-k)}$ 
with $t\in\Z_{>0}$ 
satisfies Assumption 
\ref{Ass2} (i) 
with $s(X)=\frac{-k^2-8k+128}{2(10-k)}$ for $k\leq 7$,
by Theorem 1.3 (i), we conclude 
the following. 
\begin{Cor}
For the above (partially 
compactified) 
moduli spaces of log Enriques surface with 
$k$ $\frac{1}{4}(1,1)$ singularities 
with $1\le k\le 7\ (k\neq 2)$  $X=O^+(\Lambda_{logEnr,k})\backslash \D_{L_{logEnr}}$,  
the Satake-Baily-Borel compactifications 
$\overline{X}^{\mathrm{SBB}}$ are 
Fano varieties. 
\end{Cor}
Actually, they are also unirational, by \cite{MaYos}.

\end{Ex}
\vspace{3mm}
\begin{Ex}[Simple lattices case]
Let $\Lambda$ be a quadratic lattice over $\Z$ of signature $(2,n)$.
We recall from \cite{BEF} that $\Lambda$ is called 
{\it simple} if the space of cusp forms of weight $1+\frac{n}{2}$ associated with a finite quadratic form $\Lambda^{\vee}/\Lambda$ is zero.
Then the special divisors on $\D_{\Lambda}$ are all 
given by the divisors of Borcherds lift, 
so that we can apply Theorem \ref{Fanoness}. 

In fact, Wang-Williams \cite{WW} showed that for every simple lattice $\Lambda$ of signature $(2,n)$ with $3\leq n\leq 10$, the graded algebra of modular forms for certain subgroups of the  orthogonal group is freely generated.
From this, we have the associated Shimura varieties are weighted projective spaces, in particular, 
log terminal $\Q$-Fano. 

From Theorem \ref{Fanoness}, 
all Borcherds product satisfying Assumption 
\ref{Ass2} (i) should have $s(X)>1$. 
Also from Corollary \ref{connected}, the boundary of the 
Satake-Baily-Borel compactification is in the 
closure of the branch divisors. 
See the tables of examples in \cite{WW}. 

We remark that 
before \cite{WW}, \cite{BEF} showed there are only finitely many isometry classes of even  simple lattices $\Lambda$ of signature $(2,n)$. 
\end{Ex}


\vspace{3mm}
\subsection{Preparation for unitary case - Hermitian lattice}
\label{prepare.Hlattice}
Here, we recall some material on 
Hermitian lattices treated in  
\cite{Hofmann} to prepare 
for constructing some examples of 
unitary 
modular varieties from the next 
subsection. There, 
we similarly apply 
Theorem \ref{Fanoness} to 
certain restriction of 
Borcherds products to 
explore their 
birational properties. 

Here is the setup. 
Let $F=\Q(\sqrt{d})$ be an  imaginary quadratic field for a square-free negative integer $d$, and $\OO_F$ be its ring of integers.
Let $\delta$ be the inverse different of $F$, i.e., 
\[\delta := 
\begin{cases}
\frac{1}{2\sqrt{d}}\quad(d\equiv 2,3\bmod 4),\\
\frac{1}{\sqrt{d}}\quad(d\equiv 1 \bmod 4).
\end{cases}\]
Let 
$(\Lambda,\langle\  ,\ \rangle)$ be a Hermitian lattice of signature $(1,n)$ over $\OO_F$ in the sense of \cite{Hofmann} 
i.e., a finite free $\OO_F$-module with an 
Hermitian form which is 
$\delta\OO_F$-valued. 
We define the dual lattice $\Lambda^{\vee}$ as  
\[\Lambda^{\vee} := \{v\in \Lambda\otimes_{\OO_F}F\mid \langle v,w\rangle \in\delta\OO_F\ (\forall w\in\Lambda)\}.\] 
We say $\Lambda$ is unimodular if $\Lambda=\Lambda^{\vee}$ and $\Lambda$ is even if $\langle v,v\rangle\in\Z$ for all $v\in \Lambda$.
The latter means the associated quadratic form is even.
It is also known that the quotient $A_\Lambda=\Lambda^{\vee}/\Lambda$ is a finite $\OO_F$-module, called the discriminant group. Then, 
$\widetilde{U}(\Lambda) := \{g\in U(\Lambda)\mid g\vert_
{A_\Lambda}=1_{A_\Lambda}\}$
is the so-called discriminant kernel or the stable unitary group. 
 For a Hermitian lattice $\Lambda$, we define $\Lambda(a) :=  (\Lambda,a\langle\ ,\ \rangle)$ for $a\in\delta \OO_F$.
Analogously to quadratic forms, we also have the following proposition. 

\begin{Prop}
\label{divisor_scaler_unitary}
There exists a unimodular Hermitian lattice $M$ and an element  $b\in\OO_F$ such that $\Lambda=M(b)$ if and only if the ideal $\{\langle v, w\rangle\in\delta\OO_F\mid w\in\Lambda\}$ with respect to $v\in\Lambda$ is equal $b\delta\OO_F$ for every primitive element $v\in\Lambda$. 
\end{Prop}

Let $D_\Lambda$ be the Hermitian symmetric domain (complex ball) with respect to $U(\Lambda)(\R)$, equivalently, 
\[D_\Lambda := \{v\in\mathbb{P}(\Lambda\otimes\C)\mid \l v,v\r>0\}\]
and $H(v)$ be the special divisor with respect to $v\in\Lambda$.
For any element $r\in\Lambda$ satisfying $\langle r,r\rangle <0$ and $\xi\in\OO_F^{\times}\backslash \{1\}$, we define the {\em quasi-reflection}  $\tau_{r,\xi}\in U(\Lambda)(\Q)$ with respect to $r$, $\xi$ as follows:
\[\tau_{r,\xi}(\ell) := \ell -(1-\xi)\frac{\langle \ell,r\rangle}
{\langle r,r\rangle}r.\]
Note that for $\xi=-1$, we have the usual reflection.
See also \cite{AF}. 
We also remark that, for example, for $F=\Q(\sqrt{-1})$, we get $\tau_{r,\sqrt{-1}}^2=\tau_{r,-1}$ and for $F=\Q(\sqrt{-3})$, we get $\tau_{r,\omega}^2=\tau_{r,\overline{\omega}}$ for any $r\in \Lambda$ where $\omega$ is a primitive third root of unity.

The union of ramification divisors of $\pi_{\Gamma}\colon D_\Lambda\to\Gamma\backslash D_\Lambda$ is 
  \[\bigcup_{r} H(r)\]
  by \cite[Corollary 3]{Behrens} 
  for $\Gamma\subset U(\Lambda)$ and 
  $n>1$.
  Here, the union runs thorough primitive elements  $r\in\Lambda/\OO_F^{\times}$ with $\l r,r\r<0$ such that $\eta\tau_{r,\xi}\in\Gamma$ for some $\eta\in\OO_F^{\times}$ and $\xi\in\OO_F^{\times}\backslash \{1\}$.
We consider the natural embedding of the type I domain to the type IV domain
\[\iota\colon D_\Lambda\hookrightarrow \D_{\Lambda_Q}\]
where $(\Lambda_Q,(\ ,\ ))$ is the quadratic lattice associated with $(\Lambda,\l\ ,\ \r)$, i.e., $\Lambda_Q:=\Lambda$ as a $\Z$-module and $(\ ,\ ):=\Tr_{F/\Q}\l\ ,\ \r$.
For the analysis of ramification divisors on $D_{\Lambda}$, we first prepare the following lemma.
\begin{Lem}
\label{ramification_pullback}
For $F=\Q(\sqrt{d})$, assume $d\equiv 2,3\bmod 4$ or $d=-3$.
Then \[\iota(\bigcup_{\substack{r\in\Lambda/\OO_F^{\times}\mathrm{:primitive}\\ \eta\tau_{r,\xi}\in U(\Lambda)\ \mathrm{for}\ \exists  \eta\in\OO_F^{\times},\ \exists  \xi\in\OO_F^{\times}\backslash \{1\}}}H(r))\subset\bigcup_{\substack{r\in \Lambda_Q/\pm\mathrm{:primitive}\\ \sigma_r\in O^+(\Lambda_Q) \ \mathrm{or}\  -\sigma_r\in O^+(\Lambda_Q)}}\mathcal{H}(r)\cap\iota(D_{\Lambda}).\]
\end{Lem}
\begin{proof}
For $F\neq\Q(\sqrt{-1})$, $\Q(\sqrt{-3})$, it suffices to show that if 
  \[\frac{2\l\ell,r\r}{\l r,r\r}\in\OO_F,\]
  then 
  \[\alpha:=\frac{2(\ell,r)}{(r,r)}=\frac{2\Tr_{F/_Q}\l\ell,r\r}{\Tr_{F/\Q}\l r,r\r}\in\Z.\]
      Since $\l r,r\r\in\Q$, we have 
  \[\alpha=\Re\frac{2\l\ell,r\r}{\l r,r\r}.\]
  Hence for $d\equiv 2,3\bmod 4$ with $d\neq -1$, this concludes lemma.
  
  For $F=\Q(\sqrt{-1})$, it needs to show that if
   \[(1-\sqrt{-1})\frac{\l\ell,r\r}{\l r,r\r}\in\OO_F\ \mathrm{or}\ (1+\sqrt{-1})\frac{\l\ell,r\r}{\l r,r\r}\in\OO_F,\]
  then $\alpha\in\Z$.
  In the following, let $a$, $b$ be rational integers.
  First, we assume 
  \[(1-\sqrt{-1})\frac{\l\ell,r\r}{\l r,r\r}=a+\sqrt{-1}b\in\OO_F.\]
  Then $\alpha=a-b\in\Z$.
  Second, we assume 
  \[(1+\sqrt{-1})\frac{\l\ell,r\r}{\l r,r\r}=a+\sqrt{-1}b\in\OO_F.\]
  Then $\alpha=a+b\in\Z$.
  This concludes lemma for $F=\Q(\sqrt{-1})$.
  
    For $F=\Q(\sqrt{-3})$, assume that one of the following holds.
   \begin{align}
   \label{11}
       (1\pm\omega)\frac{\l\ell,r\r}{\l r,r\r}&\in\OO_F,\\
       \label{22}
       (1\pm\omega^2)\frac{\l\ell,r\r}{\l r,r\r}&\in\OO_F,\\
       \label{33}
       2\frac{\l\ell,r\r}{\l r,r\r}&\in\OO_F.
   \end{align}
Through some simple computation, when (\ref{11}) or (\ref{22}) hold, then we have $\alpha\in\Z$.
  Finally, we assume (\ref{33}).
 Let
  \[\alpha=\alpha_1=\frac{2(\ell,r)}{(r,r)}=a-\frac{b}{2},\]
  \[\alpha_2=\frac{2(\ell,\omega r)}{(\omega r,\omega r)}=-\frac{a}{2}+b,\]
  \[\alpha_3=\frac{2(\ell,\omega^2 r)}{(\omega^2 r,\omega^2 r)}=-\frac{a+b}{2}.\]
Hence, the assumption $a+\omega b\in\OO_F$ implies one of $\alpha_i$ for $i=1,2,3$ is an element of $\Z$.
On the other hand, we have $H(r)=H(\omega r)=H(\omega^2 r)$ and $\iota(H(r))\subset\mathcal{H}(r)$, thus this concludes lemma for $F=\Q(\sqrt{-3})$.
  \end{proof}

For the computation of multiplicities of 
unitary modular forms later, we need the 
following converse to \cite[Remark after 6.1]{Hofmann}. 
\begin{Lem}\label{res}
Let $r\in \Lambda$ be a primitive element with $\l r,r\r<0$.
\begin{enumerate}
    \item \label{1}
The special divisor $H(r)$ is contained in 
exactly $\frac{\# \OO_F^{\times}}{2}$ 
special divisors of the form 
$\mathcal{H}(r')\subset \mathcal{D}_{\Lambda_Q}$ for some primitive $r'\in\Lambda_Q$. 
\item \label{reduced}
The restriction of the special divisor $\mathcal{H}_r|_{D_{\Lambda}}$ 
is $H(r)$ with multiplicity $1$ i.e., 
reduced. 
\end{enumerate}
\end{Lem}
\begin{proof}
We fix $\sqrt{d}\in \C$ and the  
corresponding embedding 
$F\hookrightarrow \C$. 
First, we prove \eqref{1}. 
Note $\mathcal{H}(r)|_{D_\Lambda}
=\mathcal{H}(r')|_{D_\Lambda}$ 
if and only if 
$\C r'=\C r$ for $r,r'\in\Lambda$. 
This implies $r=ar'$ for some $a\in\C^{\times}$.
Since $r$ is primitive, we have $a\in\OO_F^{\times}$.
On the other hand, as $\mathcal{H}(r')$ only depends on 
$\R r'$ so that $\mathcal{H}(r')
=\mathcal{H}(-r')$, 
the number we concern is 
$\frac{\# \OO_F^*}{2}$. 

The proof of \eqref{reduced} is as follows. 
Since $\l r,r\r<0$, $\mathcal{H}(r)$ 
is again an orthogonal symmetric domain 
which is an (analytic) 
open subset of a quadric 
hypersurface, say $Q^{n-1}\subset 
Q^n\subset \mathbb{P}^{n+1}$. 
Thus the 
restriction of the 
Cartier divisor $r=0$ to $Q^n$ 
is reduced and $\mathcal{H}(r)$ 
is its open subset. 
$H(r)$ is also an open subset 
of the restriction of $r=0$ to 
the linear subspace, which is also 
clearly reduced. Hence the 
assertion follows. 
\end{proof}


\subsection{Unramifiedness of unitary modular varieties}
\label{rarely_ramification_divisors}
\begin{Thm}
\label{thm:rare}
Let $F=\Q(\sqrt{d})\ (d\neq -1)$ 
be an imaginary quadratic field and  $\Lambda$ be a Hermitian unimodular  lattice over $\OO_F$ of signature $(1,n)$ for $n>1$.
We assume $d\equiv 2,3\bmod 4$. 
Then for any arithmetic subgroup $\Gamma\subset U(\Lambda)$, the canonical map $\pi_{\Gamma}\colon D_{\Lambda}\to\Gamma\backslash D_{\Lambda}$ does not ramify in codimension 1, 
so that $\overline{X}^{\rm SBB}$ 
is a log canonical model. 
\end{Thm}
\begin{proof}
It suffices to show the claim for $\Gamma= U(\Lambda)$.
The ramification divisors are defined by $\tau_{r,\xi}$ for some primitive $r\in\Lambda$ and $\xi\in\OO_F^{\times}\backslash\{1\}$ and by Lemma \ref{ramification_pullback}, they are included in the set
\[
  \bigcup_{r\in\Lambda,b\in 
  \Z,\xi\in\OO_F^{\times}\backslash\{1\}}\bigcup_{\substack{r\in \Lambda/\OO_F^{\times}\\ \langle r,r\rangle 
  =-\frac{b}{2},\  \tau_{r,\xi}\in U(\Lambda)}}H(r).\\
\]
Now 
\[\tau_{r,\xi}(\ell)=\ell-(1-\xi)\frac{\langle\ell,r\rangle}{\langle r,r \rangle}r.\]
We assume that $r\in\Lambda$ is a reflective element, that is, $\tau_{r,\xi}\in U(\Lambda)$ for some $\xi\in\OO_F^{\times}\backslash\{1\}$.
Then
\[(1-\xi)\frac{\langle\ell,r\rangle}{\langle r,r\rangle}=-\frac{2(1-\xi)\langle\ell,r\rangle}{b}.\]
Since $r$ is primitive and $\Lambda$ is unimodular, by Proposition \ref{divisor_scaler_unitary}, there exists an $\ell\in\Lambda$ such that $\langle\ell,r\rangle=\frac{1}{2\sqrt{d}}$, so we have
\[(1-\xi)\frac{\langle\ell,r\rangle}{\langle r,r\rangle}=-\frac{1-\xi}{b\sqrt{d}}\not\in\OO_F\]
for $F\neq\Q(\sqrt{-1}), \Q(\sqrt{-3})$.
This implies $\tau_{r,\xi}\not\in U(\Lambda)$ and this is contradiction.
The last assertion then follows from 
\cite{Mum77} (or as a special case of  our Theorem \ref{Fanoness} (iii)). 
\end{proof}

Note that we can also deduce this result from \cite[Lemma 2.2]{WW2}.

\begin{Cor}
Let $F=\Q(\sqrt{d})$ $(d\neq-1)$ be an imaginary quadratic field and $(\Lambda,\l\ ,\ \rangle) =M(b)$ be a Hermitian lattice over $\OO_F$ of signature $(1,n)$ for $n>1$ where $M$ is a unimodular Hermitian lattice and $b\in\OO_F$.
We assume $d\equiv 2,3\bmod 4$, and $\frac{b}{\sqrt{d}}\not\in\OO_F$. 
Then for any arithmetic subgroup $\Gamma\subset U(\Lambda)$, the canonical map $\pi_{\Gamma}\colon D_{\Lambda}\to\Gamma\backslash D_{\Lambda}$ does not ramify in codimension 1.
\end{Cor}

\subsection{Unitary modular cases, Part I - Fano cases}
Below, for the definition of Hermitian lattices; see Appendix \ref{appendix}.
\begin{Rem}
\label{unitary_a}
We can estimate the value $s(X)$ as orthogonal modular varieties and use it to determine the birational types of ball quotients.
Note that the ramification degrees arising from unitary cases may differ from orthogonal ones \cite{Behrens}, so we have to pay attention to the computation of $a$; compare with Remark \ref{rem:orthog_a}.

For $F=\Q(\sqrt{-1})$, let $B_2$ (resp. $B_4$) be a union of ramification divisor with ramification degree $2$ (resp. $4$).
If a modular form $f$ of weight $k$ vanishes on $B_2$ (resp. $B_4$) with order $2m$ (resp. $3m$) for some $m\in\Z_{>0}$, then $f$ satisfies Assumption \ref{Ass1} (i) and $s(X)=\frac{k}{4mn}$.
\end{Rem}

\vspace{3mm}
\begin{Ex}
\label{Unimodular case II_unitary}
For $F=\Q(\sqrt{-1})$, let $\Lambda:= \Lambda_{U\oplus U}\oplus \Lambda_{E_8(-1)}$ be an even unimodular  Hermitian lattice over $\OO_{\Q(\sqrt{-1})}$ of signature $(1,5)$ whose associated quadratic lattice is $\Lambda_Q=U\oplus U\oplus E_8(-1)$.

The only ramification divisors of the map $D_\Lambda\to X:= U(\Lambda)\backslash D_\Lambda$
are 
\begin{align*}
    \bigcup_{\substack{r\in \Lambda/\OO_F^{\times}\mathrm{:primitive}\\ \langle r,r\rangle=-1}}H(r)
\end{align*}
with ramification degree 2.
For more details, see  Example \ref{Unimodular case I_unitary}. 

By Example \ref{Unimodular case II}, $f:=\Phi_{252}\vert_{D_\Lambda}$ is a weight 252 modular form with 
\[{\rm div} f=2\sum_{\substack{r\in L/\OO_{\Q(\sqrt{-1})}^{\times}\mathrm{:primitive}\\ \langle r,r\rangle=-1}}H(r)\]
whose coefficient comes  from  Lemma \ref{res}.
Therefore applying Theorem \ref{Fanoness} (i) for $f^{12}$ with $s(X)=\frac{21}{2}$, we have the following.
\begin{Cor}
The Satake-Baily-Borel compactification $\overline{X}^{\mathrm{SBB}}$ of the Shimura variety $X:=U(\Lambda)\backslash D_{\Lambda}$ is
a Fano variety, where $\Lambda:= \Lambda_{U\oplus U}\oplus \Lambda_{E_8(-1)}$ for $F=\Q(\sqrt{-1})$. 
\end{Cor}
\end{Ex}

\vspace{3mm}
\begin{Ex}
\label{Example 2.3 in Odaka_unitary}
For $F=\Q(\sqrt{-1})$, let $\Lambda:= \Lambda_{U\oplus U(2)}\oplus \Lambda_{E_8(-1)}(2)$ be an even Hermitian lattice over $\OO_{\Q(\sqrt{-1})}$ of signature $(1,5)$ whose associated quadratic lattice is $\Lambda_Q=U\oplus U(2)\oplus E_8(-2)$.
The ramification divisors on $\D_{\Lambda_Q}$ with respect to $O^+(\Lambda_Q)$ is the union of special divisors with respect to $(-2)$-vectors and special-even $(-4)$-vectors, so the ramification divisors on $D_\Lambda$ with respect to $ U(\Lambda)$ are included in the union of special divisors with respect to $(-1)$-vectors and special-even $(-2)$-vectors  since $\langle v,v\rangle$ is real for all $v\in\Lambda$.
Here we say a vector $r\in\Lambda$ is special-even if $\Re\langle r,v\rangle\in\Z$ for any $v\in\Lambda$.
The only ramification divisors of $\pi$ are 
\[\bigcup_{\substack{r\in L/\OO_{\Q(\sqrt{-1})}^{\times}\mathrm{:primitive}\\ \langle r,r\rangle=-1,\  \tau_{r,-1}\in U(\Lambda)}}H(r)\quad \cup \bigcup_{\substack{r\in L/\OO_{\Q(\sqrt{-1})}^{\times}\textrm{:special-even,\  primitive}\\ \langle r,r\rangle=-2,\  \tau_{r,-1}\in U(\Lambda)}}H(r)\]
with ramification degree $d_{i}=2$ and 
    \[\bigcup_{\substack{r\in \Lambda/\OO_{\Q(\sqrt{-1})}^{\times}\mathrm{:primitive}\\ \langle r,r\rangle=-1,\  \tau_{r,\sqrt{-1}}\in U(\Lambda)}}H(r)\quad\cup\bigcup_{\substack{r\in \Lambda/\OO_{\Q(\sqrt{-1})}^{\times}\textrm{:special-even,\ primitive}\\ \langle r,r\rangle=-2,\  \tau_{r,\sqrt{-1}}\in U(\Lambda)}}H(r).
\]
with ramification degree $d_{i}=4$.
For any primitive element $r\in \Lambda$ with $\langle r,r\rangle=-1$, we have
\[\tau_{r,-1}(\ell)=\ell+2\langle \ell,r\rangle r.\]
By the description of Hermitian lattices $\Lambda_{U\oplus U(2)}$ and $\Lambda_{E_{8(-1)}(2)}$,  
\[2\l\ell,r\rangle\in\OO_{\Q(\sqrt{-1})}.\]
Hence $\tau_{r,-1}\in U(\Lambda)$ for any $(-1)$-primitive element $r\in\Lambda$.
For any special-even primitive element $r\in \Lambda$ with $\langle r,r\rangle=-2$, we have
\[\tau_{r,-1}(\ell)=\ell+\langle \ell,r\rangle r.\]
By the definition of $\Lambda_{U\oplus U(2)}$, if $\Re\langle \ell,r\rangle\in\Z$, then $\Im\langle\ell,r\rangle\in\Z$ for any $\ell\in\Lambda$.
Also by the definition of $\Lambda_{E_8}(-2)$, we have  $\l\ell,r\rangle\in\OO_{\Q(\sqrt{-1})}$ for any $\ell\in \Lambda$.
Hence $\tau_{r,-1}\in U(\Lambda)$ for any special-even $(-2)$-primitive vector $r\in \Lambda$. 
Therefore the map $D_\Lambda\to X:= U(\Lambda)\backslash D_\Lambda$ ramifies along 
\[\bigcup_{\substack{r\in L/\OO_{\Q(\sqrt{-1})}^{\times}\mathrm{:primitive}\\ \langle r,r\rangle=-1}}H(r)\quad\cup\bigcup_{\substack{r\in L/\OO_{\Q(\sqrt{-1})}^{\times}\textrm{:special-even,\ primitive}\\ \langle r,r\rangle=-2}}H(r).\]
For $(-1)$-primitive vector $r\in \Lambda$, 
\[\tau_{r,\sqrt{-1}}(\ell)=\ell+(1-\sqrt{-1})\langle\ell,r\rangle r.\]
If $r\in \Lambda_{E_8(-1)}(2)$, then by the description of the Hermitian matrix defining  $\Lambda_{E_8(-2)}$, we have $\l\ell,r\rangle\in\OO_{\Q(\sqrt{-1})}$, so $(1-\sqrt{-1})\l\ell,r\rangle\in\OO_{\Q(\sqrt{-1})}$.
If $r\in  \Lambda_{U\oplus U(2)}$, then the ideal $\{\l\ell,r\rangle\mid\ell\in \Lambda_{U\oplus U(2)}\}$ is generated by $\frac{1+\sqrt{-1}}{2}$ since $\det(L_{U\oplus U(2)})=\frac{1}{2}$, so $(1-\sqrt{-1})\l\ell,r\rangle\in\OO_{\Q(\sqrt-1)}$.
From a similar discussion as above, we have $\tau_{r,\sqrt{-1}}\in U(\Lambda)$ for any $(-1)$-primitive vector $r\in\Lambda$.

For special-even $(-2)$-primitive vector $r\in\Lambda$, 
\[\tau_{r,\sqrt{-1}}(\ell)=\ell+\frac{(1-\sqrt{-1})}{2}\l\ell,r\rangle r.\]
If $r\in \Lambda_{E_8(-1)}(2)$, then there exists an $\ell\in \Lambda_{E_8(-1)}(2)$ such that $\l\ell,r\rangle=1$, so we have $\frac{(1-\sqrt{-1})\l\ell,r\rangle}{2}=\frac{1-\sqrt{-1}}{2}\not\in\OO_{\Q(\sqrt{-1})}$.
If $r\in \Lambda_{U\oplus U(2)}$, then there exists an $\ell\in \Lambda_{U\oplus U(2)}$ such that $\l\ell,r\rangle=\frac{1+\sqrt{-1}}{2}$, so we have $\frac{(1-\sqrt{-1})\l\ell,r\rangle}{2}=\frac{1}{2}\not\in\OO_{\Q(\sqrt{-1})}$.
Thus, we have $\tau_{r,\sqrt{-1}}\not\in U(\Lambda)$ for any special-even $(-2)$-primitive vector $r\in\Lambda$.

Therefore, the ramification in codimension $1$ 
only occurs along 
\[\bigcup_{\substack{r\in \Lambda/\OO_{\Q(\sqrt{-1})}^{\times}\mathrm{:primitive}\\ \langle r,r\rangle=-1}}H(r)\]
with ramification  degree $2$, and along 
\[\bigcup_{\substack{r\in\Lambda/\OO_{\Q(\sqrt{-1})}^{\times}\textrm{:special-even,\ primitive}\\ \langle r,r\rangle=-2}}H(r)\]
with ramification  degree $4$. 

This example implies Theorem \ref{thm:rare} does not hold for non-unimodular lattices and $F=\Q(\sqrt{-1})$. 

By Example \ref{Enriques}, we have modular forms $\Phi_4\vert_{D_\Lambda}$ and $\Phi_{124}\vert_{D_\Lambda}$ such that 
 \[{\rm div} \Phi_4\vert_{D_\Lambda}=2\sum_{\substack{r\in \Lambda/\OO_{\Q(\sqrt{-1})}^{\times}\mathrm{:primitive}\\ \langle r,r\rangle=-1}}H(r)\]
\[{\rm div} \Phi_{124}\vert_{D_\Lambda}=2\sum_{\substack{r\in \Lambda/\OO_{\Q(\sqrt{-1})}^{\times}\textrm{:special-even,\ primitive}\\ \langle r,r\rangle=-2}}H(r)\]
whose coefficient again comes  from  Lemma \ref{res}.

Hence, applying Theorem \ref{Fanoness} (i) 
to $(\Phi_{4}\vert_{D_\Lambda}^2 \Phi_{124}\vert_{D_\Lambda}^3)^{12}$ with $s(X)=62$, we have the following.
\begin{Cor}
The Satake-Baily-Borel compactification $\overline{X}^{\mathrm{SBB}}$ of the Shimura variety $X:=U(\Lambda)\backslash D_{\Lambda}$
 is a Fano variety, where  $\Lambda:= \Lambda_{U\oplus U(2)}\oplus \Lambda_{E_8(-1)}(2)$ for $F=\Q(\sqrt{-1})$. 
\end{Cor}
\end{Ex}

\subsection{Unitary modular cases, Part II - 
with ample canonical class}
\begin{Ex}
\label{Unimodular case I_unitary}
For $F=\Q(\sqrt{-1})$, let $\Lambda:= \Lambda_{U\oplus U}\oplus \Lambda_{E_8(-1)}\oplus \Lambda_{E_8(-1)}\oplus\Lambda_{E_8(-1)}$ be an even unimodular  Hermitian lattice of signature $(1,13)$ whose associated quadratic lattice is $\Lambda_Q=II_{2,26}=
U\oplus U\oplus E_8(-1)\oplus E_8(-1)\oplus E_8(-1)$.
The ramification divisors on $\D_{\Lambda_Q}$ with respect to $O^+(\Lambda_Q)$ is the union of special divisors with respect to $(-2)$-vectors, so the ramification divisors on $D_{\Lambda}$ with respect to $ U(\Lambda)$ are included in the union of special divisors with respect to $(-1)$-vectors as $\langle v,v\rangle$ is real for all $v\in\Lambda$.
There exist possibly double ramification divisors i.e., those with $d_i=2$, and quadruple ramification divisors 
i.e., those with $d_i=4$, 
of the natural morphism $\pi\colon D_\Lambda\to  X:= U(\Lambda)\backslash D_\Lambda$. It ramifies in codimension $1$ along 
\[ \bigcup_{\substack{r\in \Lambda/\OO_{\Q(\sqrt{-1})}^{\times}\mathrm{:primitive}\\ \langle r,r\rangle=-1,\  \tau_{r,-1}\in U(\Lambda)}}H(r)\]
with ramification degree 2, and 
\[ \bigcup_{\substack{r\in \Lambda/\OO_{\Q(\sqrt{-1})}^{\times}\mathrm{:primitive}\\ \langle r,r\rangle =-1,\ \tau_{r,\sqrt{-1}}\in U(\Lambda)}}H(r)\]
with ramification degree 4.

For any primitive element $r\in \Lambda$ with $\langle r,r\rangle=-1$, we have
\[\tau_{r,\sqrt{-1}}(\ell)=\ell+(1-\sqrt{-1})\langle \ell,r\rangle r,\]
but by Proposition \ref{divisor_scaler_unitary} and unimodularity of $\Lambda$, $\langle\ell,r\rangle=\frac{1}{2\sqrt{-1}}$
for some $\ell\in \Lambda$.
Hence $\tau_{r,-1}\not\in U(\Lambda)$ for any $(-1)$-primitive element $r\in \Lambda$, that is, there is no quadruple ramification divisors.

For any primitive element $r\in\Lambda$ with $\langle r,r\rangle=-1$, we have
\[\tau_{r,-1}(\ell)=\ell+2\langle \ell,r\rangle r.\]
Here, 
\[\langle\ell,r\rangle\in\delta\OO_F=\frac{1}{2\sqrt{-1}}\OO_{\Q(\sqrt{-1})},\]
so $2\langle\ell,r\rangle\in\OO_{\Q(\sqrt{-1})}$.
Thus, $\tau_{r,-1}\in U(\Lambda)$ for any $(-1)$-primitive element $r\in \Lambda$, that is, there are only double ramification divisors along 

\[   \bigcup_{\substack{r\in \Lambda/\OO_{\Q(\sqrt{-1})}^{\times}\mathrm{:primitive}\\ \langle r,r\rangle=-1}}H(r)\]
with ramification degree 2.
By Example \ref{Unimodular case I}, $f:=\Phi_{12}\vert_{D_\Lambda}$ is a weight 12 modular form whose divisors are equal to double ramification divisors;
\[{\rm div} f=2\sum_{\substack{r\in \Lambda/\OO_F^{\times}\mathrm{:primitive}\\ \langle r,r\rangle=-1}}H(r)\]
whose coefficient again comes  from  Lemma \ref{res}.
Therefore applying Theorem \ref{Fanoness} (iii) to $f^{28}$ with $s(X)=\frac{3}{14}$, we have the following the following.
\begin{Cor}
The Satake-Baily-Borel compactification $\overline{X}^{\mathrm{SBB}}$ of the Shimura variety $X:=U(\Lambda)\backslash D_{\Lambda}$ is a log canonical model, where $\Lambda:= \Lambda_{U\oplus U}\oplus \Lambda_{E_8(-1)}\oplus \Lambda_{E_8(-1)}\oplus\Lambda_{E_8(-1)}$ for $F=\Q(\sqrt{-1})$. 
Recall from {\it Terminology} 
after Theorem \ref{Fanoness} that 
a log canonical model in this paper  means  
it has only log canonical singularities  
and ample canonical class. 
\end{Cor}
\end{Ex}

\vspace{3mm}
\begin{Ex}
For $F=\Q(\sqrt{-2})$, let  $\Lambda:= \Lambda'_{U\oplus U(2)}\oplus \Lambda'_{E_{8}(-1)}(2)$ be an even Hermitian lattice over $\OO_{\Q(\sqrt{-2})}$ of signature $(1,5)$.
The union of ramification divisors of the map $\pi:D_\Lambda\to X:= U(\Lambda)\backslash D_\Lambda$ are the union of special divisors with respect to $(-1)$-vectors only, unlike $F=\Q(\sqrt{-1})$ case. 
Of course, these divisors ramify with 
 ramification degree $2$, so we can also show  $\overline{X}^{\mathrm{SBB}}$ is a log canonical model.
 (Applying Theorem \ref{Fanoness} (iii) to $f^{12}$ with $s(X)=\frac{1}{6}$, where $f:=\Phi_{4}\vert_{D_\Lambda}$.)
This example implies Theorem \ref{thm:rare} does not hold for non-unimodular lattices and there exist Hermitian lattices, whose quadratic lattices are the same, admitting Shimura varieties with various birational types according to imaginary quadratic fields.
\begin{Cor}
The Satake-Baily-Borel compactification $\overline{X}^{\mathrm{SBB}}$ of the Shimura variety $X:=U(\Lambda)\backslash D_{\Lambda}$ is a log canonical model, where $\Lambda:= \Lambda'_{U\oplus U(2)}\oplus \Lambda'_{E_{8}(-1)}(2)$ for $F=\Q(\sqrt{-2})$.
\end{Cor}
\end{Ex}

\begin{Rem}
\label{-2_not_ramify}
For $F=\Q(\sqrt{-2})$, let  $\Lambda:= \Lambda'_{U\oplus U}\oplus \Lambda'_{E_8(-1)}\oplus \Lambda'_{E_8(-1)}\oplus \Lambda'_{E_8(-1)}$ be an even unimodular Hermitian lattice over $\OO_{\Q(\sqrt{-2})}$ of signature  $(1,13)$, whose associated quadratic lattice $\Lambda_Q$ is $U\oplus U\oplus E_8(-1)\oplus E_8(-1)\oplus E_8(-1)$. 

Now, we know that for any arithmetic subgroup $\Gamma\subset U(\Lambda)$, the map $\pi:D_\Lambda\to\Gamma\backslash D_\Lambda$ does not ramify in codimension $1$.
This is exactly an example of Theorem \ref{thm:rare}.
Thus the Satake-Baily-Borel compactification  $\overline{\Gamma\backslash D_\Lambda}^{\mathrm{SBB}}$ is a 
log canonical model. 
\end{Rem}

\begin{Rem}
For any imaginary quadratic field with class number 1, we can construct $\Lambda_{U\oplus U}$ and 
$\Lambda_{E_8}$; see \cite[Appendix A]{Maeda}.
As in Theorem \ref{thm:rare}, we can show that the corresponding map does not ramify in codimension $1$ for any arithmetic subgroup 
so that the Satake-Baily-Borel compactification is  log canonical model again. 
\end{Rem}


\begin{Rem}
By the same reason as Remark \ref{-2_not_ramify}, for $F\neq\Q(\sqrt{-1})$, the map $\pi:D_\Lambda\to\Gamma\backslash D_\Lambda$ does not ramify in codimension 1, where $\Lambda:= \Lambda_{U\oplus U}\oplus \Lambda_{E_8(-1)}$ and $\Gamma\subset U(\Lambda)$ is any arithmetic subgroup.
This is also an example of Theorem \ref{thm:rare} and  $\overline{\Gamma\backslash D_\Lambda}^{\mathrm{SBB}}$ is a log canonical model.
\end{Rem}


\subsection{More examples}
For $F=\Q(\sqrt{-1})$, let $\Lambda_{-1}:= \Lambda_{U\oplus U}\oplus \Lambda_{E_8(-1)}(2)$.
Then, the map $\pi:D_{\Lambda_{-1}}\to U(\Lambda_{-1})\backslash D_{\Lambda_{-1}}$ ramifies at the union of special divisors with respect to $(-1)$-vectors and $(-2)$-special-even vectors.
By \cite[Theorem 8.1]{Yoshikawa}, there exists a reflective modular form $\Psi_{12}$ of weight 12 on $\D_{(\Lambda_{-1})_Q}$ such that 
\[{\rm div} \Psi_{12}\vert_{D_\Lambda}=2\sum_{\substack{r\in \Lambda_{-1}/\OO_{\Q(\sqrt{-1})}^{\times} :\mathrm{primitive}\\ \l r,r\r=-1}}H(r)\]
whose coefficient again comes  from  Lemma \ref{res}.
Thus, $\iota^{\star}\Psi_{12}=\Psi_{12}\vert_{D_{\Lambda_{-1}}}$ is a reflective modular form on $D_{\Lambda_{-1}}$, but this does not satisfy Assumption \ref{Ass2} (ii) because the ramification divisors properly include the divisors of  $\Psi_{12}\vert_{D_{\Lambda_{-1}}}$, i.e.,
\[{\rm Supp}({\rm div} \Psi_{12}\vert_{D_\Lambda})\subsetneq \bigcup_{\substack{r\in L/\OO_{\Q(\sqrt{-1})}^{\times}\mathrm{:primitive}\\ \langle r,r\rangle=-1}}H(r)\quad\cup\bigcup_{\substack{r\in L/\OO_{\Q(\sqrt{-1})}^{\times}\textrm{:special-even,\ primitive}\\ \langle r,r\rangle=-2}}H(r),\]
where the right-hand side is the ramification divisor.
Hence, we can not show the Fano-ness of $\overline{( U(\Lambda_{-1})\backslash D_{\Lambda_{-1}})}^{\mathrm{SBB}}$ in this way (but we can show the uniruledness or more strongly, rationally-chain-connectedness of $U(\Lambda_{-1})\backslash D_{\Lambda_{-1}}$ by \cite[Theorem 5.1]{Maeda}).

On the other hand, for $F=\Q(\sqrt{-2})$, let $\Lambda_{-2}$ be the Herimtian lattice over $\OO_{\Q(\sqrt{-2})}$ of signature $(1,5)$ whose associated quadratic lattice is $U\oplus U\oplus E_8(-2)$.
Then the map $\pi:D_{\Lambda_{-2}}\to U(\Lambda_{-2})\backslash D_{\Lambda_{-2}}$ has no ramification divisors, so we can not even show the uniruledness.

\appendix
\section{Matrix definitions}
\label{appendix}
The following matrices are taken from \cite[Appendix A]{Maeda}.
\subsection{$\Q(\sqrt{-1})$ cases}
\label{app:-1}
Let $\Lambda_{U\oplus U}$ be an even unimodular  Hermitian lattice of signature $(1,1)$ over $\OO_{\Q(\sqrt{-1})}$ defined by the matrix 
\[
\frac{1}{2\sqrt{-1}}
\begin{pmatrix}
0 & 1 \\
-1 & 0 \\
\end{pmatrix}
\]
whose associated quadratic lattice $(\Lambda_{U\oplus U})_Q$ is $U\oplus U$.

Let $\Lambda_{U\oplus U(2)}$ be an even  Hermitian lattice of signature $(1,1)$ over $\OO_{\Q(\sqrt{-1})}$ defined by the matrix 
\[\frac{1}{2}
\begin{pmatrix}
0 & 1+\sqrt{-1} \\
1-\sqrt{-1} & 0
\end{pmatrix}
\]
whose associated quadratic lattice $(\Lambda_{U\oplus U(2)})_Q$ is $U\oplus U(2)$.

Let $\Lambda_{E_8(-1)}$ be an even unimodular  Hermitian lattice of signature $(0,4)$ over $\OO_{\Q(\sqrt{-1})}$ defined by the matrix
\[-
\frac{1}{2}\begin{pmatrix}
2 & -\sqrt{-1} & -\sqrt{-1}&1\\
\sqrt{-1} & 2 & 1 & \sqrt{-1}\\
\sqrt{-1} & 1 & 2 & 1\\
1& -\sqrt{-1} & 1 & 2\\
\end{pmatrix}
\]
whose associated quadratic lattice $(\Lambda_{E_8(-1)})_Q$ is $E_8(-1)$.
This matrix is called Iyanaga's matrix.

\subsection{$\Q(\sqrt{-2})$ cases}
\label{app:-2}
Let $\Lambda'_{U\oplus U}$ be an even unimodular  Hermitian lattice of signature $(1,1)$ over $\OO_{\Q(\sqrt{-2})}$ defined by the matrix
\[
\frac{1}{2\sqrt{-2}}
\begin{pmatrix}
0 & 1 \\
-1 & 0 \\
\end{pmatrix}
\]
whose associated quadratic lattice $(\Lambda'_{U\oplus U})_Q$ is $U\oplus U$.

Let $\Lambda'_{U\oplus U(2)}$ be a Hermitian lattice of signature $(1,1)$ over $\OO_{\Q(\sqrt{-2})}$ defined by the matrix
\[
\begin{pmatrix}
0 & \frac{1}{2} \\
\frac{1 }{2} & 0
\end{pmatrix}
\]
whose associated quadratic lattice $(\Lambda'_{U\oplus U(2)})_Q$ is $U\oplus U(2)$.

Let $\Lambda'_{E_8(-1)}$  be an even unimodular  Hermitian lattice of signature $(0,4)$ over $\OO_{\Q(\sqrt{-2})}$ defined by the matrix
\[
-\frac{1}{2}\begin{pmatrix}
2 & 0 & \sqrt{-2}+1&\frac{1}{2}\sqrt{-2}\\
0 & 2 & \frac{1}{2}\sqrt{-2} & 1-\sqrt{-2}\\
1-\sqrt{-2} & -\frac{1}{2}\sqrt{-2} & 2 & 0\\
-\frac{1}{2}\sqrt{-2}& \sqrt{-2}+1 & 0 & 2\\
\end{pmatrix}
\]
whose associated quadratic lattice $(\Lambda'_{E_8(-1)})_Q$ is $E_8(-1)$.

\subsection*{Acknowledgements}
Y.M would like to express his gratitude to his adviser, Tetsushi Ito, for his helpful comments and warm encouragement. 
We would also like to thank Shouhei Ma for  constructive discussion on quadratic and Hermitian lattices and thank Ken-ichi Yoshikawa 
for the discussion on modular forms. 
Y.M is supported by JST ACT-X JPMJAX200P.
Y.O is partially supported by 
KAKENHI 18K13389, 20H00112 and 16H06335.



\begin{thebibliography}{FGA}

\bibitem[Ale96]{Alexeev}
V. Alexeev, 
Log canonical singularities and complete moduli of stable 
pairs, arXiv:9608013 

 \bibitem[AF02]{AF}
  D. Allcock, E. Freitag, 
  Cubic surfaces and Borcherds products,
  Comment.\ Math.\ Helv.\ 77 (2002),\ no.\ 2,\ 270-296.
  

\bibitem[ACT02]{ACT}
D. Allcock, J. A. Carlson, D. Toledo, 
The complex hyperbolic geometry of the moduli space of cubic surfaces, J. Algebraic Geom. 11 (2002), 659-724.

\bibitem[Amb03]{Amb}
F. Ambro, Quasi-log varieties, Tr. Mat. Inst. Steklova 240 (2003), Biratsion. Geom. Linein. Sist. Konechno Porozhdennye Algebry, 220-239; translation in Proc. Steklov Inst. Math. 2003, no. 1 (240), 214-233.


\bibitem[AMRT]{AMRT}
A. Ash, D. Mumford, M. Rapoport, Y. -S. Tai, Smooth compactification of locally symmetric varieties, Cambridge Mathematical Library. 

\bibitem[BB66]{BB}
W. Baily, A. Borel, 
Compactification of arithmetic quotients of bounded symmetric domains, Ann. of Math. (2) 84 (1966), 442–528.

\bibitem[BN94]{BN}
W. Barth, I. Nieto, 
Abelian surfaces of type $(1,3)$ 
and quartic surfaces with 
$16$ skew lines, 
J.\ Alg.\ Geom.\ 3 (1994), 
173-222. 

\bibitem[Beh12]{Behrens}
  N. Behrens, 
  Singularities of ball quotients,
  Geom.\ Dedicata\ 159 (2012),\ 389-407. 
  
  






  \bibitem[Bor95]{infinite}
  R. Borcherds, 
  Automorphic forms on $O_{s+2,2}(\R)$ and infinite products,
  Invent.\ Math.\ 120 (1995),\ no.\ 1,\ 161-213.
  
  
\bibitem[Bor96]{Borcherds.Enriques}
R. Borcherds, 
The moduli space of Enriques surfaces and 
the fake Monster Lie superalgebra, 
Topology 35 (1996), no. 3, 699–710.


\bibitem[Bor98]{Borcherds.general}
R. Borcherds, Automorphic forms with singularities on Grassmannians, Invent. Math. 132 (1998), no. 3, 491–562.

\bibitem[BGLM21]{BGLM}
L. Braun, D. Greb, K .Langlois, J. Moraga, 
Reductive quotients of klt singularities, 
arXiv:2111.02812. 

\bibitem[Bru02]{Bruinier.book}
J. H. Bruinier, 
Borcherds products on $O(2,l)$ and Chern classes of Heegner divisors.
Lecture Notes in Mathematics, 1780. Springer-Verlag, Berlin, 2002.

\bibitem[Bru14]{Bruinier.article}
J. H. Bruinier, 
On the converse theorem for Borcherds products, Jour.\ of Algebra 397 (2014), 315-342.

  \bibitem[BEF16]{BEF}
  J. H. Bruinier, S. Ehlen, E. Freitag, 
Lattices with many Borcherds products,
Math. Comp. 85 (2016), no. 300, 1953–1981.



\bibitem[DY20]{DY}
 X. Dai, K. Yoshikawa, Analytic torsion for log-Enriques surfaces and Borcherds product, arXiv:2009.10302.


\bibitem[DKW19]{DKW}
 C. Dieckmann, A. Krieg, M. Woitalla, The graded ring of modular forms on the Cayley half-space of 
degree two, Ramanujan J. 48 (2019), no. 2, 385-398.

\bibitem[DN89]{DN}
J. -M. Drezet, M. S. Narasimhan, 
Groupe de Picard des vari\'et\'es de modules de fibr\'es semi-stables sur les courbes alg\'ebriques, 
Invent. Math. 97 (1989), no. 1, 53-94. 

\bibitem[Fre83]{Freitag}
E. Freitag,
Siegelsche Modulfunktionen,
Springer-Verlag, Berlin, 1983.

\bibitem[FSM10]{FSM}
E. Freitag, R. Salvati Manni, 
Some Siegel threefolds 
with a Calabi-Yau model, 
Ann. Scuola Norm. Sup. Pisa 
Cl. Sci. (5) vol.\ (2010), 
833-850. 

\bibitem[Fjn10]{Fjn.Fano}
O. Fujino, 
Some problems on Fano varieties, 
Proceeding to the conference 
``Fukuso-kikagaku no shomondai (2010)" in Japanese. 



\bibitem[Fjn11]{Fjn}
O. Fujino, Fundamental theorems for the log minimal model program, Publ. Res. Inst. Math. Sci. 47 (2011), no. 3, 727-789. 

\bibitem[Fjn17]{Fjn.MSJ}
O. Fujino, Foundations of the minimal model program, 
MSJ Memoirs, 35. Mathematical Society of Japan, Tokyo, 2017.

\bibitem[FG12]{FG}
O. Fujino, Y. Gongyo, On canonical bundle formulas and subadjunctions, Michigan Math. J. 61 (2012), no. 2, 255-264. 


  



\bibitem[Gri10]{Gr}
V. Gritsenko, 
Reflective modular forms in algebraic geometry, 
arXiv:1012.4155. 

\bibitem[Gri12]{GB}
V. Gritsenko, 
$24$ faces of the Borcherds 
modular forms $\Phi_{12}$, arXiv:1203.6503. 

\bibitem[Gri18]{Gri}
V. Gritsenko, 
Reflective modular forms and applications, 
Russian Math. Surv. 73 (2018) 797-864. 

\bibitem[GH14]{GrHu}
V. Gritsenko, K. Hulek,
Uniruledness of orthogonal modular varieties, 
J. Alg. Geom. 23 (2014), 711-725. 

  \bibitem[GH16]{GH}
  V. Gritsenko, K. Hulek,
  Moduli of polarized Enriques surfaces,
  in K3 surfaces and their moduli,\ 55-72,\ Progr.\ Math.,\ 315,\ 2016.
  
    \bibitem[GHS07]{GHS}
V. Gritsenko, K. Hulek, G. Sankaran, 
The Kodaira dimension of the moduli spaces of K3 surfaces, Invent. math. 169 (2007), 519-567.



\bibitem[GN98]{GrN}
V. Gritsenko, V. Nikulin, 
Automorphic forms and Lorentzian Kac-Moody algebras. I, Internat. J. Math. 9 (1998), no. 2, 153-199.

\bibitem[GN18]{GN}
V. Gritsenko, V. Nikulin, 
Lorentzian Kac-Moody algebras with Weyl groups of 2-reflections. Proc. Lond. Math. Soc. (3) 116 (2018), no. 3, 485-533.

\bibitem[HM07]{HM}
C. Hacon, J. Mckernan, 
On Shokurov's rational connectedness conjecture, 
Duke Math. J. 138 (1) 119 - 136, 15 (2007). 

\bibitem[HU22]{HashimotoUeda}
K. Hashimoto, K. Ueda, 
The ring of modular forms for the even unimodular lattice of signature 
$(2,10)$, 
Proc. Amer. Math. Soc. 150 (2022), no. 2, 547–558.

\bibitem[Hof14]{Hofmann}
E. Hofmann, Borcherds products on unitary groups. Math. Ann. 358 (2014), no. 3-4, 799-832.

\bibitem[Huy94]{Huy}
D. Huybrechts, 
Complete curves in moduli spaces of stable bundles on surfaces, 
Math. Ann. 298 (1994), no. 1, 67-78. 

\bibitem[Ich09]{Ichikawa}
T. Ichikawa, 
Siegel modular forms of 
degree $2$ over rings, 
J. Number Theory 129 (2009), no. 4, 818–823.

\bibitem[Igu64]{Igusa}
 J. Igusa, On Siegel modular forms of genus two. II, Amer. J. Math. 86 (1964), 392–412.
 
\bibitem[KM98]{KM}
J. Koll\'ar, S. Mori, 
Birational geometry of algebraic varieties, 
Cambridge Tracts in Mathematics, 134. Cambridge University Press, Cambridge, 1998.








\bibitem[Kon93]{KondoI}
S. Kond\={o}, 
On the Kodaira dimension of the moduli spaces of K3 surfaces, 
Compositio Math.\ 89 (1993), 251-299. 

\bibitem[Kon94]{Kondo.Enriques}
S. Kond\={o}, 
The rationality of the moduli space of Enriques surfaces, 
Compositio Math. 91 (1994), 159-173. 

\bibitem[Kon99]{KondoII}
S. Kond\={o}, 
On the Kodaira dimension of the moduli spaces of K3 surfaces. II, 
Compositio Math.\ 116 (1999), 111-117. 

\bibitem[Kon02]{Kondo02}
S. Kond\={o}, 
The moduli space of Enriques surfaces and 
Borcherds products, 
J.\ Alg.\ Geom.\ 11 (2002), 601-627. 

\bibitem[Lej93]{Lej}
P. Lejarraga, 
The moduli of Weierstrass fibrations over $\mathbb{P}^{1}$: rationality, Rocky Mountain J. Math. 23 (1993), no. 2, 649–650. 

\bibitem[Li94]{JLi}
J. Li, 
Kodaira dimension of moduli space of vector bundles on surfaces, 
Invent. Math. 115 (1994), no. 1, 1-40. 

\bibitem[Loo84]{Looijenga}
E. Looijenga, 
The Smoothing Components of a Triangle Singularity. II, Math. Ann. 269 (1984), no. 3, 357–387.


\bibitem[Ma12]{MaYos}
S. Ma, 
The unirationality of the moduli spaces of 2-elementary K3 surfaces. With an appendix by Ken-Ichi Yoshikawa. Proc. Lond. Math. Soc. (3) 105 (2012), no. 4, 757-786.


\bibitem[Ma18]{Ma.gentype}
S. Ma, 
On the Kodaira dimension of orthogonal modular varieties, 
Invent.\ Math.\ (2018), pp. 859-911. 

\bibitem[Mae20]{Maeda}
Y. Maeda, 
Uniruledness of unitary Shimura varieties 
associated with Hermitian forms of signatures 
$(1,3), (1,4), (1,5)$. 
arXiv:2008.13106, Master thesis (Department of Mathematics, 
Kyoto university). 

   \bibitem[Mae22]{Maeda_big}
  Y. Maeda, 
  Big line bundles on unitary modular varieties, arXiv:2204.01128. 



\bibitem[Mum77]{Mum77}
D. Mumford, 
Hirzebruch's proportionality theorem in the noncompact case, Invent Math. 42 (1977), 239-277. 

\bibitem[Mum83]{Mumford}
D. Mumford,
On the Kodaira dimension of the Siegel modular variety,
Algebraic geometry—open problems, 348-375, Lecture Notes in Math., 997, Springer, Berlin, 1983.

\bibitem[NU22]{NU}
A. Nagano, K. Ueda, 
The ring of modular forms for the even unimodular lattice of signature $(2,18)$, Hiroshima Math. J. 52 (2022), no. 1, 43-51.

\bibitem[Nam85]{Namikawa}
Y. Namikawa, 
Periods of Enriques surfaces, Math. Ann. 270 (1985), no. 2, 201–222.

\bibitem[OSS16]{OSS}
Y. Odaka, C. Spotti, S. Sun, 
Compact moduli spaces of del Pezzo surfaces and K\"ahler-Einstein metrics, J. Differential Geom. 102 (2016), no. 1, 127-172.


\bibitem[OO21]{OO18}
Y. Odaka, Y. Oshima, 
Collapsing K3 surface, Tropical geometry 
and Moduli compactifications of Satake, Morgan-Shalen type, 
MSJ Memoirs, 40. Mathematical Society of Japan, Tokyo, 2021.

\bibitem[Od20]{Od20}
Y. Odaka, 
PL density invariant for type II degenerating K3 surfaces, Moduli compactification and hyperKahler metrics, Nagoya 
Math. Journal (2020). 

\bibitem[Sat60]{Sat}
I. Satake, 
On compactifications of the quotient spaces for arithmetically 
defined discontinuous groups, Ann. of Math. (2) 72 (1960), 555-580. 

\bibitem[Ste91]{Sterk}
H. Sterk, 
Compactifications of the period space of Enriques  surfaces I, Math. Z. 207 (1991), no. 1, 1–36.

\bibitem[Tai82]{Tai}
Y. S. Tai, On the Kodaira dimension of the moduli space of abelian
varieties, Invent. Math. 68 (1982), 425-439. 

\bibitem[vdG82]{vdG}
G. van der Geer, 
On the geometry of a Siegel modular threefold.
Math. Ann. 260 (1982), no. 3, 317-350. 

\bibitem[vdG21]{vdG2}
G. van der Geer, 
Siegel modular forms of 
degree two and three 
and invariant theory, 
arXiv:2102.02245. 

\bibitem[Yos09]{Yoshikawa2}
K. Yoshikawa, 
Calabi-Yau threefolds of Borcea-Voisin, analytic torsion, 
and Borcherds products, Astérisque No. 328 (2009), 355–393 (2010).



  \bibitem[Yos13]{Yoshikawa}
  K, Yoshikawa,
  \textit{K3 surfaces with involution, equivalent analytic torsion, and automorphic forms on the moduli spaces, II: A structure theorem for $r(M)>10$},
  J.\ Reine Angew.\ Math.\ 677 (2013),\ 15-70.



\bibitem[WW20]{WW}
H. Wang, B. Williams, 
Simple lattices and free algebras of modular forms, 
arXiv:2009.13343. 

\bibitem[WW21]{WW2}
H. Wang, B. Williams, 
Free algebras of modular forms on ball quotients, 
arXiv:2105.14892. 

\bibitem[Zha91]{Zhang.logEnriques}
D-Q. Zhang, 
Logarithmic Enriques surfaces, 
J. Math. Kyoto Univ. 31 (1991), 419-466. 


\bibitem[Zha06]{Zhang}
Q. Zhang, 
Rational connectedness of log $Q$-Fano varieties, 
J. Reine Angew. Math. 590 (2006) 131-142.


\end{thebibliography}
\end{document}